\documentclass[10pt,a4paper,reqno,oneside]{amsart}  %%(fold)
\usepackage{amssymb,graphicx,mathrsfs,xcolor}
\usepackage[colorlinks=true, linkcolor=magenta, citecolor=cyan, urlcolor=blue]{hyperref}

\numberwithin{equation}{section}\newtheorem{theorem}{Theorem}[section]
\newtheorem{corollary}[theorem]{Corollary}\newtheorem{lemma}[theorem]{Lemma}

\newtheorem{proposition}[theorem]{Proposition}\theoremstyle{remark}
\newtheorem{remark}{Remark}[section]\newtheorem*{ack}{Acknowledgments}
\theoremstyle{definition}

\newcommand{\bra}[1]{\langle #1 \rangle}
\newcommand{\one}[1]{\mathbf{1}_{#1}}

\title[Nonlinear Dirac equation]%
{Endpoint estimates and global existence
for the nonlinear Dirac equation with potential}
\date{\today}

\author{Federico Cacciafesta}
\address{Federico Cacciafesta: 
SAPIENZA --- Universit\`a di Roma,
Dipartimento di Matematica, 
Piazzale A.~Moro 2, I-00185 Roma, Italy}
\email{cacciafe@mat.uniroma1.it}

\author{Piero D'Ancona}
\address{Piero D'Ancona: 
SAPIENZA --- Universit\`a di Roma,
Dipartimento di Matematica, 
Piazzale A.~Moro 2, I-00185 Roma, Italy}
\email{dancona@mat.uniroma1.it}
\subjclass[2000]{%
}\keywords{}

\begin{document}\maketitle
  \begin{abstract}
    We prove endpoint estimates with angular regularity for
    the wave and Dirac equations perturbed with a small
    potential. The estimates are applied to prove
    global existence for the cubic Dirac equation perturbed
    with a small potential,
    for small initial $H^{1}$ data with additional angular regularity.
    This implies in particular global existence in the critical 
    energy space $H^{1}$ for small radial data.
  \end{abstract}
%%%%%   INDEX (toc)
%\tableofcontents
%(end)

\section{Introduction}\label{sec:introduction}  %(fold)

The main topic of this paper is the 
cubic massless Dirac equation on $\mathbb{R}^{1+3}$ 
perturbed with a potential
\begin{equation}\label{eq:maindirac}
  iu_{t}=\mathcal{D}u+V(x)u+P_{3}(u,\overline{u}),\qquad
  u(0,x)=f(x),
\end{equation}
where 
$u(t,x):\mathbb{R}_{t}\times \mathbb{R}^{3}_{x}\to \mathbb{C}^{4}$
and $P_{3}(u,\overline{u})$ is any homogeneous,
$\mathbb{C}^{4}$-valued cubic polynomial.
We recall that the Dirac operator $\mathcal{D}$ is defined as
\begin{equation*}
  \mathcal{D}=i^{-1}(\alpha_{1}\partial_{1}+
  \alpha_{2}\partial_{2}+\alpha_{3}\partial_{3})
\end{equation*}
where $\partial_{j}$ are the partial derivatives on $\mathbb{R}^{3}_{x}$
and $\alpha_{j}$ are the Dirac matrices
\begin{equation}\label{eq.diracm}
  \alpha_{1}=
  \begin{pmatrix}
    0 & 0 & 0 & 1 \\
    0 & 0 & 1 & 0 \\
    0 & 1 & 0 & 0 \\
    1 & 0 & 0 & 0
  \end{pmatrix},\quad \alpha_{2}=
  \begin{pmatrix}
    0 & 0 & 0 & -i \\
    0 & 0 & i & 0 \\
    0 & -i & 0 & 0 \\
    i & 0 & 0 & 0
  \end{pmatrix},\quad \alpha_{3}=
  \begin{pmatrix}
    0 & 0 & 1 & 0 \\
    0 & 0 & 0 & -1 \\
    1 & 0 & 0 & 0 \\
    0 & -1 & 0 & 0
  \end{pmatrix}.
\end{equation}
The anticommutation relations
\begin{equation*}
  \alpha_{\ell}\alpha_{k}+\alpha_{k}\alpha_{\ell}
     =2\delta_{kl}I_{4}
\end{equation*}
imply that $\mathcal{D}^{2}=-I_{4}\Delta$ is
a diagonal operator, showing the intimate connection of
the massless Dirac system with the wave equation.

The unperturbed nonlinear Dirac equation
\begin{equation}\label{eq:unperturbed}
  iu_{t}=\mathcal{D}u+F(u),\qquad u(0,x)=f(x)
\end{equation}
is important in relativistic quantum mechanics, and
was studied in a number of works 
(see e.g.~ 
\cite{Reed76-a},
\cite{DiasFigueira87-a},
\cite{Najman92-a},
\cite{Moreau89-a},
\cite{EscobedoVega97-a},
\cite{MachiharaNakamuraOzawa04-a},
\cite{MachiharaNakamuraNakanishi05-a}
and for the more general Dirac-Klein-Gordon system see
\cite{DanconaFoschiSelberg07-a},
\cite{DanconaFoschiSelberg07-b}).
In particular,
it is well known that the cubic nonlinearity is critical
for solvability in the energy space $H^{1}$;
global existence in $H^{1}$
is still an open problem even for small initial data, while
the case of subcritical spaces $H^{s}$, $s>1$
was settled in the positive in
\cite{EscobedoVega97-a},
\cite{MachiharaNakamuraOzawa04-a}.

Criticality is better appreciated in terms of Strichartz
estimates, which are the main tool in the study of
nonlinear dispersive equations. 
The identity
\begin{equation*}%\label{eq:reprdirac}
  e^{it \mathcal{D}}f=
  \cos(t|D|)f+ i \frac{\sin(t|D|)}{|D|}\mathcal{D}f,
  \qquad
  |D|=(-\Delta)^{1/2}
\end{equation*}
shows that the estimates for the Dirac flow $e^{it \mathcal{D}}$
are immediate consequences of the corresponding estimates
for the wave flow $e^{it|D|}$, with the same indices, 
restricted to the
special case of dimension $n=3$.
For the wave equation on $\mathbb{R}^{1+n}$, $n\ge3$,
Strichartz estimates can be combined with
Sobolev embedding and take the general form
\begin{equation}\label{eq:freeStrn}
  \||D|^{\frac nr+\frac1p-\frac n2}
    e^{it |D|}f\|_{L^{p}L^{r}}\lesssim \|f\|_{L^{2}}
\end{equation}
for all $p,r$ such that
\begin{equation*}
  p\in[2,\infty],\qquad
  0<\frac1r\le\frac12-\frac{2}{(n-1)p}.
\end{equation*}
We are using here the mixed time-space $L^{p}L^{q}$ norms defined by
\begin{equation*}
  \|u(t,x)\|_{L^{p}L^{q}}=
  \left\|
    \|u\|_{L^{q}_{x}}
  \right\|_{L^{p}_{t}}.
\end{equation*}
Notice that the limiting case $r=\infty$
\begin{equation}\label{eq:endpfalsen}
  \|e^{it|D|}f\|_{L^{2}L^{\infty}}\lesssim
  \||D|^{\frac{n-1}{2}}f\|_{L^{2}}
\end{equation}
is always excluded and is indeed false for general data.
See \cite{GinibreVelo95-a} and \cite{KeelTao98-a}
for the general Strichartz estimates; concerning the
limiting case $r=\infty$, see 
\cite{KlainermanMachedon93-a},
\cite{FangWang06-a}.
The corresponding estimates for the Dirac equation are given in
\cite{DanconaFanelli07-a}.

In particular, the endpoint estimate
\begin{equation}\label{eq:endpfalse}
  \|e^{it \mathcal{D}}f\|_{L^{2}L^{\infty}}\lesssim
  \||D|f\|_{L^{2}}
\end{equation}
fails. To see the connection with the critical equation
\eqref{eq:unperturbed}, we rewrite it as
a fixed point problem for the map
\begin{equation*}
  v\mapsto \Phi(v)=e^{it \mathcal{D}}f
  +i\int_{0}^{t}e^{i(t-t')\mathcal{D}}P_{3}(v(t'))dt'.
\end{equation*}
If \eqref{eq:endpfalse} were true one could write
\begin{equation*}
  \left\|
    \int_{0}^{t}e^{i(t-t')\mathcal{D}}v(t')^{3}dt
  \right\|_{L^{2}L^{\infty}}
  \lesssim
  \int_{-\infty}^{+\infty}
  \|e^{it \mathcal{D}}e^{-it' \mathcal{D}}P_{3}(v(t'))\|
  _{L^{2}L^{\infty}}dt'
  \lesssim
  \|v^{3}\|_{L^{1}H^{1}}
\end{equation*}
and in conjuction with the conservation of $H^{1}$ energy this
would imply
\begin{equation*}
  \|\Phi(v)\|_{L^{\infty}_{t}H^{1}_{x}}+
     \|\Phi(v)\|_{L^2_{t}L^{\infty}_{x}}\lesssim
  \|f\|_{H^{1}}+\|v\|_{L^{\infty}H^{1}}\|v\|_{L^{2}L^{\infty}}^{2}.
\end{equation*}
In other words, a contraction argument in the norm
$\|\cdot\|_{L^{2}L^{\infty}}+\|\cdot\|_{L^{\infty}H^{1}}$ would be
enough to prove global existence of small $H^{1}$ solutions to
\eqref{eq:unperturbed}.

It was already noted in \cite{KlainermanMachedon93-a} 
that \eqref{eq:endpfalsen} is true for radial data when $n=3$. 
This remark is not of immediate 
application for the Dirac equation, since solutions
corresponding to radial data need not be radial
(due to the fact that the operator
$\mathcal{D}$ does not commute with rotations of $\mathbb{R}^{3}$).
Nevertheless, for radial $H^{1}$
and even more general data, in
\cite{MachiharaNakamuraNakanishi05-a}
global existence was achieved via finer 
estimates, which separate radial from angular regularity. 
We introduce the natural notations
\begin{equation*}
  \|f\|_{L^{a}_{r}L^{b}_{\omega}}=
  \left(
    \int_{0}^{\infty}\|f(r\ \cdot\ )\|_{L^{b}(\mathbb{S}^{n-1})}
    ^{a}r^{n-1}dr
  \right)^{\frac1a}
\end{equation*}
and
\begin{equation*}
  \|f\|_{L^{\infty}_{r}L^{b}_{\omega}}=
  \sup_{r\ge0}
    \|f(r\ \cdot\ )\|_{L^{b}(\mathbb{S}^{n-1})}.
\end{equation*}
Then the following estimate is proved
in \cite{MachiharaNakamuraNakanishi05-a}:
\begin{equation}\label{eq:mnnop}
  n=3,\qquad
  \|e^{it |D|}f\|
      _{L^{2}L^{\infty}_{r}L^{p}_{\omega}}\lesssim
  \sqrt{p}\cdot
  \||D|f\|_{L^{2}},\qquad \forall p<\infty.
\end{equation}
This gives a bound for the standard $L^{2}L^{\infty}$ norm
via Sobolev embedding on the unit sphere $\mathbb{S}^{2}$
\begin{equation}\label{eq:refined}
  \|e^{it |D|}f\|_{L^{2}L^{\infty}}\lesssim
  \|\Lambda^{\epsilon}_{\omega} e^{it |D|}f\|
      _{L^{2}L^{\infty}_{r}L^{p}_{\omega}}
  \lesssim
  \||D|\Lambda^{\epsilon}_{\omega}f\|_{L^{2}}, 
  \qquad
  p>\frac{2}{\epsilon}
\end{equation}
where the angular derivative operator $\Lambda_{\omega}^{s}$
is defined in terms of the Laplace-Beltrami operator on 
$\mathbb{S}^{n-1}$ as
\begin{equation*}
  \Lambda^{s}_{\omega}=(1-\Delta_{\mathbb{S}^{n-1}})^{s/2}.
\end{equation*}
Using \eqref{eq:refined}
one can prove global existence for \eqref{eq:unperturbed}
provided the norm $\||D|\Lambda^{s}_{\omega}f\|_{L^{2}}$ of the data
is small enough for some $s>0$. In particular, this includes all
radial data with a small $H^{1}$ norm.

Our main goal here is to extend this group of results to the
equation \eqref{eq:maindirac} perturbed with a small potential
$V(x)$. We consider first the linear equation
\begin{equation}\label{eq:linearV}
  iu_{t}=\mathcal{D}u+V(x)u+F(t,x).
\end{equation}
The perturbative term $Vu$ can not be handled using the
inhomogeneous version of \eqref{eq:mnnop} because of the
loss of derivatives. Instead, we prove new mixed 
Strichartz-smoothing estimates
(Theorem \ref{the:strichnonom})
\begin{equation}\label{eq:ourmixed}
  n\ge3,\qquad
  \left\|
    \int_0^t e^{i(t-s)|D|}F(s,x)ds
  \right\|_{L^2_tL^\infty_{|x|}L^2_\omega}
  \lesssim
  \|
    \bra{x}^{\frac{1}{2}+}|D|^{\frac{n-1}{2}}\Lambda_{\omega}^{\sigma}F
  \|_{L^2_tL^2_{x}}
\end{equation}
where
\begin{equation}\label{eq:sigma-i}
\begin{split}
  \ \text{\textbf{for $n=3$,}}&\qquad \sigma=0
    \\
  \ \text{\textbf{for $n\ge4$,}}&\qquad \sigma=1-\frac n2.
\end{split}
\end{equation}

\begin{remark}\label{rem:comparehomog}
  As a byproduct of our proof, we obtain
  the following endpoint estimates for the wave flow
  with gain of angular regularity
  (Theorem \ref{the:strichom}):
  \begin{equation}\label{eq:ourwave}
    n\ge3,\qquad
    \| e^{it|D|}f\|_{L^2_tL^\infty_{r}L^2_\omega}
    \lesssim
    \|\Lambda^{\sigma}_{\omega} f\|_{\dot H^\frac{n-1}{2}}
  \end{equation}
  where $\sigma$ is as in \eqref{eq:sigma-i}.
  Although this was not the main purpose of the paper,
  it is interesting to compare
  \eqref{eq:ourwave} with known results. In dimension $n=3$,
  estimate \eqref{eq:ourwave} is just a special case of 
  Theorem 1.1-III in \cite{MachiharaNakamuraNakanishi05-a} 
  where \eqref{eq:ourwave} is proved with
  $\sigma=-\frac34$; it is not known if this value
  is sharp, however in the same paper it is proved that
  the estimate is false for $\sigma<-\frac56$.
  On the other hand, to our knowledge,
  estimate \eqref{eq:ourwave} for $n\ge4$ and \eqref{eq:ourmixed}
  for $n\ge3$ are new. The literature on these kind
  of estimates is extensive and we refer to
  \cite{FangWang08-a}, \cite{JiangWangYu10-a}
  and the references therein for further information.
\end{remark}

Combining \eqref{eq:ourmixed} with the techniques
of \cite{DanconaFanelli08-a} we obtain the following endpoint
result for a 3D linear wave equation with singular potential.
Analogous estimates can be proved for higher dimensions; here we
chose to focus on the 3D case since the 
assumptions on $V$ take a particular simple form:

\begin{theorem}\label{the:strichWE-i}
  Let $n=3$ and consider the Cauchy problem for the wave equation
  $$u_{tt}-\Delta u+V(x)u=F,\qquad
  u(0,x)=f(x),\qquad u_{t}(0,x)=g(x)$$
  under the assumptions:
  \begin{enumerate}\setlength{\itemindent}{-10pt}
    \renewcommand{\labelenumi}{(\roman{enumi})}
    \item $V(x)$ is real valued and
    the positive and negative parts $V_{\pm}$ satisfy
    \begin{equation}\label{eq:assV-i}
      V_{+}\le \frac{C}{|x|^{\frac12-\epsilon}+|x|^{2}},\qquad
      V_{-}\le \frac{\delta}{|x|^{\frac12-\epsilon}+|x|^{2}}
    \end{equation}
    for some $\delta,\epsilon$ sufficiently small and some $C\ge0$;
    \item $-\Delta+V$ is selfadjoint;
    \item 0 is not a resonance for $-\Delta+V_{-}$
    (in the following sense: if $f$ is such that $(-\Delta+V_-)f=0$ 
    and $\bra{x}^{-1}f\in L^{2}$, then $f\equiv0$).
  \end{enumerate}
  Then the solution $u(t,x)$
  satisfies the endpoint Strichartz estimate
  \begin{equation}\label{eq:endWEV-i}
    \| u\|_{L^2_tL^\infty_{r}L^2_\omega}
    \lesssim
    \|f\|_{\dot H^{1}}
    +
    \|g\|_{L^{2}}
    +
    \|\bra{x}^{\frac12+}
      F\|_{L^{2}_{t}L^{2}_{x}}.
  \end{equation}
\end{theorem}

The next step is to prove suitable smoothing estimates
for the Dirac equation with potential
\begin{equation*}
  iu_{t}=\mathcal{D}u+V(x)u+F(t,x)
\end{equation*}
(see Proposition \ref{pro:smooD} and Corollary \ref{cor:nablau}).
Then by a perturbative argument we
obtain the following endpoint estimates for the linear flows:

\begin{theorem}\label{the:strichD-i}
  Assume that the hermitian matrix $V(x)$ satisfies,
  for $\delta$ sufficiently small, $C$ arbitrary and $\sigma>1$, with
  $v_{\sigma}(x)=|x|^{\frac12}|\log|x||^{\sigma}+\bra{x}^{1+\sigma}$,
  \begin{equation}\label{eq:assnablaV2-i}
    |V(x)|\leq\frac{\delta}
         {v_{\sigma}(x)},\qquad
    |\nabla V(x)|\leq\frac{C}
         {v_{\sigma}(x)}.    
  \end{equation}
  Then the perturbed Dirac flow
  satisfies the endpoint Strichartz estimate
  \begin{equation}\label{eq:enddiracV-i}
    \|e^{it(\mathcal{D}+V)}f\|_{L^2_t L^\infty_{r}L^2_\omega}
    \lesssim
    \| f\|_{H^1}.
  \end{equation}
  If the potential satisfies the stronger assumptions: 
  for some $s>1$,
  \begin{equation}\label{eq:nablaangV2-i}
    \|\Lambda_{\omega}^{s}
         V(|x|\ \cdot\ )\|_{L^{2}(\mathbb{S}^{2})}
    \le \frac{\delta}{v_{\sigma}(x)},
    \qquad
    \|\Lambda_{\omega}^{s}\nabla
         V(|x|\ \cdot\ )\|_{L^{2}(\mathbb{S}^{2})}
    \le \frac{C}{v_{\sigma}(x)},
  \end{equation}
  then we have the endpoint estimate with angular regularity
  \begin{equation}\label{eq:enddiracVang-i}
    \|\Lambda^{s}_{\omega}
      e^{it(\mathcal{D}+V)}f\|_{L^2_t L^\infty_{r}L^2_\omega}
    \lesssim
    \|\Lambda^{s}_{\omega} f\|_{H^1}
  \end{equation}
  and the energy estimate with angular regularity
  \begin{equation}\label{eq:energyang-i}
    \|\Lambda^{s}_{\omega}
      e^{it(\mathcal{D}+V)}f\|_{L^\infty_t H^{1}}
    \lesssim
    \|\Lambda^{s}_{\omega} f\|_{H^1}
  \end{equation}
\end{theorem}

We can finally apply Theorem \ref{the:strichD-i} to the
nonlinear equation \eqref{eq:maindirac} and we obtain:

\begin{theorem}\label{the:globalNL-i}
  Consider the perturbed Dirac system \eqref{eq:maindirac}, where
  the $4\times4$ matrix valued potential $V(x)$
  is hermitian and satisfies assumptions
  \eqref{eq:nablaangV2-i}.
  Let $P_{3}(u,\overline{u})$ be a $\mathbb{C}^{4}$-valued 
  homogeneous cubic polynomial.
  Then for any $s>1$ there exists $\epsilon_{0}$ such that
  for all initial data satisfying
  \begin{equation}\label{eq:data}
    \|\Lambda_{\omega}^{s}f\|_{H^{1}}<\epsilon_{0}
  \end{equation}
  the Cauchy problem \eqref{eq:NLD} admits a unique global solution
  $u\in CH^{1}\cap L^{2}L^{\infty}$
  with $\Lambda^{s}_{\omega}u\in L^{\infty}H^{1}$.
\end{theorem}

In particular, problem \eqref{eq:maindirac} 
has a global unique solution for all radial
data with sufficiently small $H^{1}$ norm.

\begin{remark}\label{rem:wave}
  It is clear that our methods can also be applied to
  nonlinear wave equations perturbed with potentials, and
  allow to prove global well posedness for some types of
  critical nonlinearities. This problem will be the object
  of a further note.
\end{remark}

\begin{remark}\label{rem:generality}
  We did not strive for the sharpest condition on the
  potential $V$, which can be improved
  at the price of additional technicalities which
  we prefer to skip here.  
  Moreover, the result can be extended
  to more general cubic nonlinearities $|F(u)|\sim |u|^{3}$.
  
  Notice also that we need an angular regularity
  $s>1$ on the data, higher than the $s>0$ assumed
  in  the result of \cite{MachiharaNakamuraNakanishi05-a}. 
  It is possible to relax our assumptions to $s>0$;
  the only additional tool we would need to prove is
  a Moser-type product estimate
  \begin{equation*}
    \|\Lambda^{s}_{\omega}(uv)\|_{L^{2}_{\omega}}\lesssim
    \|u\|_{L^{\infty}_{\omega}}
    \|\Lambda^{s}_{\omega}v\|_{L^{2}_{\omega}}
    +
    \|\Lambda^{s}_{\omega}u\|_{L^{2}_{\omega}}
    \|v\|_{L^{\infty}_{\omega}},\qquad
    s>0
  \end{equation*}
  and an analogous one for $\Lambda^{s}_{\omega}|D|(uv)$.
  This would require a fair amount of calculus on the sphere
  $\mathbb{S}^{2}$, and here we preferred to use the
  conceptually
  much simpler algebra property of $H^{s}(\mathbb{S}^{n-1})$
  for $s>\frac{n-1}{2}$.
  
  On the other hand, the extension of our results to the
  massive case
  \begin{equation*}
    iu_{t}=\mathcal{D}u+V(x)u+m\beta u +F(u),\qquad
    m\neq0
  \end{equation*}
  requires a different approach and will be the object of
  further work.
\end{remark}

\begin{ack}
  We are indebted with
  Ilia Krasikov and Jim Wright
  for invaluable conversations  which helped shape up the
  proofs in Section \ref{sec:estimates}.
\end{ack}

% section introduction (end)

\section{Endpoint estimates for the free flows}\label{sec:estimates}  %(fold)

To fix our notations, we recall some basic facts on spherical
harmonics (see \cite{SteinWeiss71-a}) on $\mathbb{R}^{n}$, $n\ge2$.
For $k\geq0$, we denote by $\mathcal{H}_k$ the space of harmonic
polynomials homogeneous of degree $k$, restricted to the unit sphere
$\mathbb{S}^{n-1}$. The dimension of $\mathcal{H}_{k}$ for $k\ge2$ is
\begin{equation*}
  d_{k}=
  \binom{n+k-1}{k}-\binom{n+k-3}{k-2}\simeq \bra k ^{n-2}
\end{equation*}
while $d_{0}=1$ and $d_{1}=n$.
$\mathcal{H}_{k}$ is called the space of \emph{spherical harmonics of
degree $k$}, and we denote by $Y_{k}^{l}$, $1\le l\le d_{k}$ an
orthonormal basis. Since
\begin{equation*}
  L^2(\mathbb{S}^{n-1})=\bigoplus_{k=0}^\infty \mathcal{H}^k
\end{equation*}
every function $f(x)=f(r\omega)$, $r=|x|$, can be expanded as
\begin{equation}\label{eq:spherexp}
  f(r)=
  \sum_{k=0}^\infty\sum_{l=1}^{d_{k}} f^{l}_{k}(r)Y_k^l(\omega)
\end{equation}
and we have
\begin{equation*}
  \|f(r \omega)\|_{L^{2}_{\omega}}=
  \sum_{k\ge0\atop 1\le l\le d_{k}}
  |f_{k}^{l}|^{2},
\end{equation*}
where we use the notation $L^{2}_{\omega}=L^{2}(\mathbb{S}^{n-1})$.
More generally, if $\Delta_{\mathbb{S}}$ is the Laplace-Beltrami
operator on $\mathbb{S}^{n-1}$ and 
\begin{equation*}
  \Lambda_{\omega}=(1-\Delta_{\mathbb{S}})^{1/2},
\end{equation*}
we have the equivalence
\begin{equation*}
  \|\Lambda_{\omega}^{\sigma} f(r \omega)\|_{L^{2}_{\omega}}\simeq
  \sum_{k\ge0\atop 1\le l\le d_{k}}\bra{k}^{2 \sigma}
  |f_{k}^{l}|^{2},\qquad \sigma\in \mathbb{R}.
\end{equation*}
As a consequence we have the equivalence
\begin{equation}\label{eq:equiv1}
  \|\Lambda_{\omega}^{\sigma}f\|_{L^{2}(\mathbb{R}^{n})}^{2}\simeq
  \sum_{k\ge0\atop 1\le l\le d_{k}}\bra{k}^{2 \sigma}
  \|f_{k}^{l}(r)r^{\frac{n-1}{2}}\|
    _{L^{2}_{r}(0,\infty)}^{2}.
\end{equation}
In a similar way
\begin{equation}\label{eq:equiv2}
\begin{split}
  \|\nabla f\|_{L^{2}(\mathbb{R}^{n})}^{2} & =
  (-\Delta f,f)_{L^{2}}\simeq
    \\
  \simeq 
  \sum_{k\ge0\atop 1\le l\le d_{k}}  &
  \left(
  \|r^{\frac{n-1}{2}}\partial_{r} f_{k}^{l}(r)\|
    _{L^{2}_{r}(0,\infty)}^{2}
  +k^{2}
  \|r^{\frac{n-3}{2}} f_{k}^{l}(r)\|
    _{L^{2}_{r}(0,\infty)}^{2}
  \right)
\end{split}
\end{equation}
where we used the following representation of the action of $\Delta$
\begin{equation*}
  -\Delta f (x)= \sum Y^{l}_{k}\left(\frac{x}{|x|}\right)
  \left[-
    r^{1-n}\partial_{r}(r^{n-1}\partial_{r}f^{l}_{k})
    +
    \frac{k(k+n-2)}{r^{2}}f^{l}_{k}
  \right],\qquad r=|x|
\end{equation*}
More generally we have for integer $m$
\begin{equation*}
  -\Delta (1-\Delta_{S})^{m} f (x)
  = \sum (1+k(k+n-2))^{m}
  Y^{l}_{k}
  \left[-
    r^{1-n}\partial_{r}(r^{n-1}\partial_{r}f^{l}_{k})
    +
    \frac{k(k+n-2)}{r^{2}}f^{l}_{k}
  \right]
\end{equation*}
which implies
\begin{equation}\label{eq:equiv3}
\begin{split}
  \|\nabla \Lambda_{\omega}^{m} f\|_{L^{2}(\mathbb{R}^{n})}^{2} & =
  (-\Delta(1-\Delta_{S})^{m} f,f)_{L^{2}}\simeq
    \\
  \simeq 
  \sum_{k\ge0\atop 1\le l\le d_{k}}  &
  \bra{k}^{2m}
  \left(
  \|r^{\frac{n-1}{2}}\partial_{r} f_{k}^{l}(r)\|
    _{L^{2}_{r}(0,\infty)}^{2}
  +k^{2}
  \|r^{\frac{n-3}{2}} f_{k}^{l}(r)\|
    _{L^{2}_{r}(0,\infty)}^{2}
  \right)
\end{split}
\end{equation}
and by interpolation and duality we see that \eqref{eq:equiv3}
holds for all $m\in \mathbb{R}$.

We shall estimate the solution using the following norm:
\begin{equation*}
  \|f\|_{L^{\infty}_{r}L^{2}_{\omega}}=
  \sup_{r>0}\|f(r \omega)\|_{L^{2}_{\omega}(\mathbb{S}^{n-1})}.
\end{equation*}

\begin{theorem}\label{the:strichom}
  For all $n\geq 4$ the following estimate holds:
  \begin{equation}\label{eq:strichartz1}
    \| e^{it|D|}f\|_{L^2_tL^\infty_{r}L^2_\omega}
    \lesssim
    \|\Lambda^{1-\frac n2}_{\omega} f\|_{\dot H^\frac{n-1}{2}},
  \end{equation}
  while for $n=3$ we have
  \begin{equation}\label{eq:strich3D}
    \| e^{it|D|}f\|_{L^2_tL^\infty_{r}L^2_\omega}
    \lesssim
    \|f\|_{\dot H^1}
  \end{equation}
\end{theorem}

\begin{remark}\label{rem:3D}
  In dimension $n=3$ the previous result is a special case of the
  stronger estimate proved in \cite{MachiharaNakamuraNakanishi05-a}:
  \begin{equation}\label{eq:mnno}
    \| e^{it|D|}f\|_{L^2_tL^\infty_{r}L^2_\omega}
    \lesssim
    \|\Lambda^{-3/4}_{\omega} f\|_{\dot H^1}.
  \end{equation}
  Notice that it is not known if estimate \eqref{eq:mnno} is sharp.
  For higher dimension, estimate \eqref{eq:strichartz1}
  seems to be new; it is reasonable to guess that this result
  is not sharp and might be improved at least to
  \begin{equation}\label{eq:strichartz1conj}
    \| e^{it|D|}f\|_{L^2_tL^\infty_{r}L^2_\omega}
    \lesssim
    \|\Lambda^{\epsilon-\frac {n-1}2}_{\omega} f\|_{\dot H^\frac{n-1}{2}},
    \qquad \epsilon>0.
  \end{equation}
\end{remark}

\begin{proof}%[of ...]
  It is well known that the $\mathcal{H}_{k}$ spaces are invariant
  for the Fourier transform $\mathcal{F}$, and more precisely
  \begin{equation}\label{eq:trasfspher}
    \mathcal{F}\left(c(r) Y^l_k(\omega)\right)(\xi)=
    g(|\xi|)Y^l_k\left(\frac{\xi}{|\xi|}\right)
  \end{equation}
  where $g$ is given by the Hankel transform
  \begin{equation}\label{eq:coeff}
    g(r)=(2\pi)^{\frac{n}{2}}i^{-k}r^{-\frac{n-2}{2}}
    \int_0^\infty  c(\rho)J_{k+\frac{n-2}{2}}(r\rho)
    \rho^\frac{n}{2}d\rho.
  \end{equation}
  Here $J_\nu$ is the \emph{Bessel function} of order $\nu$
  which we shall represent using the Lommel integral form
  \begin{equation}\label{eq:bessrep}
  J_\nu(y)=\displaystyle\frac{(y/2)^\nu}{\pi^{\frac12}\Gamma(\nu+1/2)}
  \int_{-1}^1e^{iy \lambda}(1-\lambda^2)^{\nu-\frac{1}{2}}d\lambda.
  \end{equation}
  Now, given a function $f(x)$, we denote by $\check{f}$ its
  inverse Fourier transform and with $\check{f}^l_k(r)$ 
  the coefficients of the expansion in spherical harmonics
  of $\check{f}$:
  \begin{equation}\label{eq:sphdat}
    \check{f}=\sum_{k=0}^\infty\sum_{l=1}^{d_{k}}
    \check{f}^l_k(r)Y^l_k(\omega).
  \end{equation}
  Recalling (\ref{eq:trasfspher}) we obtain the representation
  \begin{equation}\label{eq:rep}
    f(x)=
     \sum(2\pi)^\frac{n}{2}i^{-k}|x|^{1-\frac{n}{2}}
     Y^l_k\left(\frac{x}{|x|}\right)
     \int_0^\infty\check{f}_k^l(\rho)J_{k+\frac{n-2}{2}}(|x|\rho)
     \rho^\frac{n}{2}d\rho
  \end{equation}
  which implies
  \begin{equation}\label{eq:waveprop}
    e^{it|D|}f=
    \sum(2\pi)^\frac{n}{2}i^{-k}|x|^{1-\frac{n}{2}}
    Y^l_k\left(\frac{x}{|x|}\right)
    \int_0^\infty e^{it\rho}\check{f}_k^l(\rho)
    J_{k+\frac{n-2}{2}}(|x|\rho)\rho^\frac{n}{2}d\rho.
  \end{equation}
  Consider now Lommel's formula \eqref{eq:bessrep}
  for $J_{\nu}$; since
  $e^{i\lambda y}=(iy)^{-k}\partial ^k_\lambda(e^{i\lambda y})$,
  after $k$ integration by parts we obtain
  \begin{equation}\label{eq:bessel2}
    J_{k+\frac{n-2}{2}}(y)=
      c_k y^{\frac{n}{2}-1}\int_{-1}^1e^{i\lambda y}
      \partial^k_\lambda\left((1-\lambda^2)^{k+\frac{n-3}{2}}\right)
      d\lambda
  \end{equation}
  with
  \begin{equation}\label{eq:constant}
  c_k=\frac{i^{k}2^{-\frac{n}{2}-k+1}}
      {\pi^{\frac12}\Gamma(\frac{n-1}{2}+k)}.
  \end{equation}
  Thus we can write
  \begin{equation}\label{eq:rapp}
  \begin{split}
    |x|^{1-\frac{n}{2}}
    &\int_0^\infty e^{it\rho} 
      \check{f}^l_k(\rho) J_{k+\frac{n-2}{2}}(|x|\rho)
      \rho^\frac{n}{2}d\rho=
    \\
      &= c_k \int_{-1}^1 \partial^k_\lambda
      \left((1-\lambda^2)^{k+\frac{n-3}{2}}\right)
      \left[\int_{-\infty}^{+\infty}\one{+}(\rho) \check{f}^l_k(\rho) \rho^{n-1}
      e^{i\rho(t+\lambda |x|)}d\rho\right]d\lambda
  \end{split}
  \end{equation}
  where $\one{+}(\rho)$ is the characteristic function of
  $(0,+\infty)$; regarding the inner integral as a Fourier
  transform we arrive at
  \begin{equation*}
    =c_k\int_{-1}^1\partial^k_\lambda
        \left((1-\lambda^2)^{k+\frac{n-3}{2}}\right)    
        \widehat{g}^l_k(t+\lambda|x|)d\lambda
  \end{equation*}
  where 
  \begin{equation}\label{eq:gkl}
    g^l_k(\rho)=\one{+}(\rho)\check{f}^l_k(\rho)\rho^{n-1}
  \end{equation}
  In conclusion, we have the following representation
  \begin{equation}\label{eq:waveprop2}
    e^{it|D|}f=
      \sum(2\pi)^\frac{n}{2}i^{-k}
      Y^l_k\left(\frac{x}{|x|}\right)
      c_k\int_{-1}^1\partial^k_\lambda
      \left((1-\lambda^2)^{k+\frac{n-3}{2}}\right)     
      \widehat{g}^l_k(t+\lambda|x|)d\lambda
  \end{equation}
  where the constants $c_k$ are given by \eqref{eq:constant} and 
  $g^l_k$ by \eqref{eq:gkl}. 
  Notice that similar representations play a fundamental role
  also in \cite{FangWang08-a}, \cite{JiangWangYu10-a}
  In particular this gives for
  the $L^{2}_{\omega}$ norm of the solution at $t,|x|$ fixed
  the formula
  \begin{equation}\label{eq:angnorm}
    \|e^{it|D|}f(|x|\cdot)\|_{L^2_\omega}^2
    \simeq 
    \sum|c_{k}|^{2}
    \left|\int_{-1}^1\partial^k_\lambda
          \left((1-\lambda^2)^{k+\frac{n-3}{2}}\right)     
      \widehat{g}^l_k(t+\lambda|x|)d\lambda\right|^2.
  \end{equation}
  We now need the following estimate:
  
  \begin{lemma}\label{lem:stimapartial}
    Let $Q_{k}(x)$ be the function
    \begin{equation*}
      Q_{k}(x)=
      \frac{\partial_x^k\left((1-x^2)^{k+\frac{n-3}{2}}\right)}
               {2^k\Gamma(k+\frac{n-1}{2})}.
    \end{equation*}
    Then we have on $x\in[-1,1]$
    \begin{equation}\label{eq:stimapar}
      |Q_{k}(x)|\lesssim\bra{k}^{1-\frac n2}
      \ \text{if $n\ge4$},\qquad
      |Q_{k}(x)|\le 1
      \ \text{if $n=3.$}
    \end{equation}
  \end{lemma}
  
  \begin{proof}
    We recall that the Jacobi polynomials are defined by
    \begin{equation}\label{eq:Jac}
      \bold{P}^{(\alpha,\beta)}_k(x)
         =\frac{(-1)^k}{2^k k!}
         (1-x)^{-\alpha}(1+x)^{-\beta}
         \frac{d^k}{dx^k}
         \left[(1-x)^{\alpha+k}(1+x)^{\beta+k}\right].
    \end{equation}
    We shall use some standard properties of these polynomial
    which can be found in \cite{AbramowitzStegun64-a}.
    The function $Q_{k}$ can be expressed in terms of
    $\bold{P}^{(\alpha,\alpha)}_k(x)$ with $\alpha=(n-3)/2$ as
    \begin{equation}\label{eq:jacob}
      |Q_{k}(x)|= 
      \frac{k!(1-x^2)^\frac{n-3}{2}}{\Gamma(k+\frac{n-1}{2})}
      \left|\bold{P}^{(\frac{n-3}{2},\frac{n-3}{2})}_k(x)\right|.
    \end{equation}
  Thus in order to estimate $Q_{k}$ we need a bound for the function
  \begin{equation}\label{eq:Tx}
  T_a(x)=(1-x^2)^a\bold{P}^{(a,a)}_k(x),\qquad
  a=\frac{n-3}{2}.
  \end{equation}
  
  The following approach was suggested by Ilia Krasikov, see
  \cite{Krasikov05-a}. Consider the second orderd differential equation
  \begin{equation*}
    f''(x)+p(x)f'(x)+q(x)f(x)=0
  \end{equation*}
  on the interval $(-1,1)$, and define the \emph{Sonine function} as
  \begin{equation*}
    S(f,x)=f(x)^2+\displaystyle\frac{f'(x)^2}{q(x)}
  \end{equation*}
  under the assumption $q>0$. It is easy to
  check that function $S$ satisfies the relation
  \begin{equation*}
    S'=-\left(2\frac{p}{q}+\frac{q'}{q^2}\right)f'^2.
  \end{equation*}
  The function $T_a(x)$ defined in \eqref{eq:Tx} satisfies the 
  differential equation
  \begin{equation*}
    T_a''(x)+\frac{2(2a-1)}{1-x^2}x T_a'(x)
       +\frac{(k+1)(2a+k)}{1-x^2}T_a(x)=0
  \end{equation*}
  so that the associated Sonine function 
  \begin{equation*}
    S_{a}(x)=T_{a}^{2}+\frac{1-x^{2}}{(k+1)(2a+k)}{T'_{a}}^{2}
  \end{equation*}
  satisfies
  \begin{equation}\label{eq:sonine}
    S_a'=\displaystyle-\frac{2(2a-1)}{(k+1)(2a+k)}x\:T_a'^2.
  \end{equation}
  From this identity it is clear that $S_{a}$ has a maximum at
  $x=0$ provided $a\ge1/2$ i.e.~$n\ge4$. In this case we have
  \begin{equation*}
    S_{a}(x)\le
    S_a(0)=T_a(0)^2+\frac{T_a'(0)^2}{(k+1)(2a+k)}=
        \bold{P}_{k}^{(a,a)}(0)^2+
        \frac{{\bold{P}_{k}^{(a,a)}}'(0)^2}{(k+1)(2a+k)}
  \end{equation*}
  Now we recall that, for \emph{even} $k\ge2$,
  \begin{equation*}
    \bold{P}_{k}^{(a,a)}(0)=
    \frac{\Gamma\left(k+a+1\right)}
         {(-2)^{k}
         \Gamma\left(\frac k2+1\right)
         \Gamma\left(\frac k2+a+1\right)}
         \simeq(-1)^{k} k^{-\frac12}
    \qquad
    {\bold{P}_{k}^{(a,a)}}'(0)=0
  \end{equation*}
  where we used the Stirling asymptotics
  \begin{equation*}
    k!\simeq k^{k-1/2}e^{-k},\qquad
    \Gamma(k+a+1)\simeq k^{k+a-1/2}e^{-k}.
  \end{equation*}
  In a similar way, for \emph{odd} $k$,
  \begin{equation*}
    \bold{P}_{k}^{(a,a)}(0)=0,
    \qquad
    {\bold{P}_{k}^{(a,a)}}'(0)=
    \frac{\Gamma\left(k+a+1\right)}
         {(-2)^{k-1}
         \Gamma\left(\frac k2+\frac12\right)
         \Gamma\left(\frac k2+a+\frac12\right)}
    \simeq (-1)^{k-1}k^{\frac12}.
  \end{equation*}
  Thus for all values of $k\ge1$ we have
  \begin{equation*}
    |T_{a}(x)|\le \sqrt{S_{a}(x)}
    \lesssim \frac1{\sqrt{k}}
  \end{equation*}
  and by \eqref{eq:jacob} we conclude that, for $k\ge1$ and $|x|<1$,
  \begin{equation*}%\label{eq:estQ}
    |Q_{k}(x)|\lesssim k^{1-\frac n2}
  \end{equation*}
  which is precisely \eqref{eq:stimapar} for $n\ge4$.
  
  In the remaining case $n=3$ we have $a=0$ and the best we can
  do is to use the sharp inequality $|\bold{P}_{k}^{{(0,0)}}|\le1$
  to obtain
  \begin{equation*}
    |Q_{k}(x)|=\frac{k!}{k!}||\bold{P}_{k}^{{(0,0)}}|\le1.
  \end{equation*}
 \end{proof}

  Using the Lemma, we can continue estimate \eqref{eq:angnorm} as
  follows
  \begin{equation*}
    \|e^{it|D|}f(|x|\cdot)\| _{L^2_\omega}^2 \lesssim
    \sum \omega_{k}^{2}
    \left(\int_{-1}^1|\widehat{g}^l_k(t+\lambda|x|)|d\lambda\right)^2
  \end{equation*}
  where
  \begin{equation}\label{eq:omegak}
    \omega_{k}=1 \quad\text{if $n=3$},\qquad
    \omega_{k}=\bra{k}^{1-\frac n2}
    \quad\text{if $n\ge4$}.
  \end{equation}
  Since
  \begin{equation*}
    \int_{-1}^1|\widehat{g}^l_k(t+\lambda|x|)|d\lambda=
    \frac{1}{|x|}\int_{-|x|}^{|x|}|\widehat{g}^l_k(t+\lambda)|d\lambda
    \le M(\widehat{g}^l_k)(t)
  \end{equation*}
  where $M(g)$ is the centered maximal function, we obtain
  \begin{equation*}
    \|e^{it|D|}f(|x|\cdot)\| _{L^2_\omega}^2 \lesssim
    \sum \omega_{k}^{2} M(\widehat{g}^l_k)(t)^{2}.
  \end{equation*}
  Now we can take the sup in $|x|$ which gives
  \begin{equation*}
    \|e^{it|D|}f\| _{L^{\infty}_{r} L^2_\omega}^2 \lesssim
    \sum \omega_{k}^{2} M(\widehat{g}^l_k)(t)^{2},
  \end{equation*}
  and integrating in time, by the $L^{2}$ boundedness of the
  maximal funcion, we obtain
  \begin{equation*}%\label{eq:almost}
    \|e^{it|D|}f\| _{L^{2}_{t}L^{\infty}_{r} L^2_\omega}^2 \lesssim
    \sum \omega_{k}^{2} \|\widehat{g}^l_k\|^{2}_{L^{2}}\simeq
    \sum \omega_{k}^{2} \|g^l_k\|^{2}_{L^{2}}\simeq
    \sum \omega_{k}^{2} \|\check{f}^l_k(\rho)\rho^{n-1}\|^{2}
         _{L^{2}_{\rho}(0,\infty)}.
  \end{equation*}
  It is immediate to check that the last sum is equivalent to
  \begin{equation*}
    \sum \omega_{k}^{2} \|\check{f}^l_k(\rho)\rho^{n-1}\|^{2}
         _{L^{2}_{\rho}(0,\infty)}\simeq
    \||D|^{\frac{n-1}{2}}\Lambda_{\omega}^{\sigma}f\|^{2}
       _{L^{2}(\mathbb{R}^{n})}
  \end{equation*}
  where $\sigma=1-n/2$ for $n\ge4$, which proves \eqref{eq:strichartz1},
  and $\sigma=0$ for $n=3$, which proves \eqref{eq:strich3D}. 
\end{proof}

Although the method of proof of
Theorem \ref{the:strichom} is probably not sharp for the homogeneous
operator, it has the advantage that it can be adapted to handle
also the nonhomogeneous term and gives the following
mixed Strichartz-smoothing estimate:

\begin{theorem}\label{the:strichnonom}
  For any $n\ge3$, the following estimate holds:
  \begin{equation}\label{eq:strichartz2}
    \left\|
      \int_0^t e^{i(t-s)|D|}F(s,x)ds
    \right\|_{L^2_tL^\infty_{|x|}L^2_\omega}
    \lesssim
    \|
      \bra{x}^{\frac{1}{2}+}|D|^{\frac{n-1}{2}}\Lambda_{\omega}^{\sigma}F
    \|_{L^2_tL^2_{x}}
  \end{equation}
  where
  \begin{equation*}
    \sigma=1-\frac n2 \quad\text{if $n\ge4$},\qquad
    \sigma=0 \quad\text{if $n=3$}.
  \end{equation*}
\end{theorem}

\begin{proof}
  As in the proof of the previous theorem, we expand $F$ in 
  spherical harmonics and we obtain the representation
  \begin{equation}\label{eq:nonomorap}
  \begin{split}
    \int_0^t e^{i(t-s)|D|}F & (s,x) ds=
    \\
    =\sum(2\pi)^{\frac n2} &  i^{-k}c_{k}
      Y^l_k\left(\frac{x}{|x|}\right)
      \int_{-1}^1\partial^k_\lambda
      \left((1-\lambda^2)^{k+\frac{n-3}{2}}\right)     
      \widehat{G}^l_k(s,t-s+\lambda|x|)d\lambda
  \end{split}
  \end{equation}
  with the constants $c_k$ as in \eqref{eq:constant}, where the functions
  $G^{l}_{k}$ are defined as follows:
  denoting by $F^{l}_{k}(t,r)$ the coefficients of the expansion
  into spherical harmonics of the inverse Fourier transform 
  $\check F=\mathcal{F}^{-1}(F)$
  \begin{equation*}
    \check F(s,x)=
    \sum \check F_{k}^{l}(s,|x|)Y^{l}_{k}\left(\frac{x}{|x|}\right)
  \end{equation*}
  and by $G^{l}_{k}$ the functions
  \begin{equation}\label{eq:Gkl}
    G^{l}_{k}(s,\rho)=\one{+}(\rho)
       \rho^{n-1}\check{F}^l_k(s,\rho),
  \end{equation}
  the $\widehat{G}^{l}_{k}(s,r)$ are the Fourier transforms of $G^{l}_{k}$
  in the second variable:
  \begin{equation*}
    \widehat{G}^{l}_{k}(s,r)=\int_{-\infty}^{+\infty}
      e^{ir\rho}G^{l}_{k}(s,\rho)d\rho.
  \end{equation*}
  Thus applying Lemma \ref{lem:stimapartial} we obtain
  \begin{equation}\label{eq:modnon}
    \left|\int_0^t e^{i(t-s)|D|}F(s,x)ds\right|\lesssim
    \sum |Y^l_k|   %\left(\frac{x}{|x|}\right)
    \frac{\omega_{k}}{|x|}
    \int_{-|x|}^{|x|}d \lambda \int_0^t ds
    |\widehat{G}^l_k(s,t-s+\lambda)|
  \end{equation}
  where $\omega_k$ is the same as in \eqref{eq:omegak}.
  We estimate the integral in $s$ as follows
  \begin{equation*}
  \begin{split}
    \int_{0}^{t}|\widehat{G}^{l}_{k}|ds
    \le &
    \int_{-\infty}^{+\infty}
      \bra{\lambda+t-s}^{\frac12+}\bra{\lambda+t-s}^{-\frac12-}
      |\widehat{G}^{l}_{k}(s,\lambda+t-s)|ds
    \\
    \lesssim &
    \left(
    \int
      \bra{\lambda+t-s}^{1+}
      |\widehat{G}^{l}_{k}(s,\lambda+t-s)|^{2}ds
    \right)^{\frac12}=
    Q^{l}_{k}(\lambda+t),
  \end{split}
  \end{equation*}
  where
  \begin{equation*}
    Q^{l}_{k}(\mu)=
      \left(\int_{-\infty}^\infty |\widehat G^l_k(s,\mu-s)|^2
      \langle \mu-s\rangle^{1+}ds\right)^\frac{1}{2}.
  \end{equation*}
  Thus we see that
  \begin{equation*}
    \frac{1}{|x|}
      \int_{-|x|}^{|x|}d \lambda \int_0^t ds
      |\widehat{G}^l_k(s,t-s+\lambda)|\lesssim
      \frac{1}{|x|}
        \int_{-|x|}^{|x|}Q^{l}_{k}(\lambda+t)d \lambda
      \le M(Q^{l}_{k})(t).
  \end{equation*}
  Coming back to \eqref{eq:modnon} we obtain
  \begin{equation*}
    \left|\int_0^t e^{i(t-s)|D|}F(s,x)ds\right|\lesssim
    \sum \omega_{k}|Y^{l}_{k}|  %\left(\frac{x}{|x|}\right)
    M(Q^{l}_{k})(t)
  \end{equation*}
  and taking first the $L^{2}_{\omega}$ norm, then the sup in $|x|$,
  then the $L^{2}_{t}$ norm, by the $L^{2}$ boundedness of the
  maxiaml function we have
  \begin{equation*}
    \left\|
        \int_0^t e^{i(t-s)|D|}F(s,x)ds\right\|^{2}
    _{L^{2}_{t}L^{\infty}_{r}L^{2}_{\omega}}
    \lesssim
    \sum \omega_{k}^{2}\|Q^{l}_{k}(t)\|_{L^{2}_{t}}^{2}.
  \end{equation*}
  The definition of $Q^{l}_{k}$ implies
  \begin{equation*}
    \int |Q^{l}_{k}(t)|^{2}dt=
    \iint|\widehat{G}^{l}_{k}(s,\mu-s)|^{2}\bra{\mu-s}^{1+}
    dsd\mu=
    \|\widehat{G}^{l}_{k}(t,r)\bra{r}^{\frac12+}\|^{2}
    _{L^{2}_{t}L^{2}_{r}}
  \end{equation*}
  and hence
  \begin{equation}\label{eq:almostdone}
    \left\|
        \int_0^t e^{i(t-s)|D|}F(s,x)ds\right\|^{2}
    _{L^{2}_{t}L^{\infty}_{r}L^{2}_{\omega}}
    \lesssim
    \sum \omega_{k}^{2}
      \|\widehat{G}^{l}_{k}(t,r)\bra{r}^{\frac12+}\|^{2}
      _{L^{2}_{t}L^{2}_{r}}.
  \end{equation}
  Recalling the definition \eqref{eq:Gkl} of $G^{l}_{k}$,
  we see that to obtain \eqref{eq:strichartz2} 
  it is sufficient to prove the
  following general inequality for $s=1/2+$ and
  arbitrary $\sigma$:
  \begin{equation}\label{eq:genineq}
    \sum \bra{k}^{2 \sigma}
      \left\|\bra{y}^{s}\mathcal{F}_{\lambda\to y}\left(
          \one{+}(\lambda)\lambda^{n-1}
          \check f^{l}_{k}(\lambda)\right)\right\|^{2}_{L^{2}_{y}}
    \lesssim\|\bra{x}^{s}\Lambda^{\sigma}_{\omega}
         |D|^{\frac{n-1}{2}}f\|
              ^{2}_{L^{2}(\mathbb{R}^{b})}.
  \end{equation}
  Here as usual $\check f^{l}_{k}$ denotee the coefficients
  in the expansion in spherical harmonics of the inverse Fourier
  transform $\check f= \mathcal{F}^{-1}f$.
  
  First of all, since 
  $\mathcal{F}^{-1}(|D|^{\frac{n-1}{2}}f)=|\xi|^{\frac{n-1}{2}}\check f$,
  we see that it is enough to prove, for $0\le s\le 1$
  and arbitrary $\sigma$, the slightly simpler
  \begin{equation}\label{eq:genineq2}
    \sum \bra{k}^{2 \sigma}
      \left\|\bra{y}^{s}\mathcal{F}_{\lambda\to y}\left(
          \one{+}(\lambda)\lambda^{\frac{n-1}{2}}
          \check f^{l}_{k}(\lambda)\right)\right\|^{2}_{L^{2}_{y}}
    \lesssim\|\bra{x}^{s}\Lambda^{\sigma}_{\omega}f\|
              ^{2}_{L^{2}(\mathbb{R}^{b})}.
  \end{equation}
  The inequality will follow by interpolation between the cases
  $s=0$ and $s=1$; indeed, we can regard it as the statement that
  the operator $T$ defined as
  \begin{equation*}
    T:f\mapsto \left\{
       \bra{y}^{s}\mathcal{F}_{\lambda\to y}\left(
           \one{+}(\lambda)\lambda^{\frac{n-1}{2}}
       \check g^{l}_{k}\right)
       \right\}_{l,k},\qquad g=\Lambda^{-\sigma}_{\omega}f
  \end{equation*}
  which associates to the function $f$ the sequence of coefficients
  in the expansion of $\mathcal{F}^{-1}(\Lambda^{-\sigma}_{\omega}f)$,
  multiplied by $\lambda^{(n-1)/2}\one{+}$, transformed again and
  multiplied by $\bra{y}^{s}$, is bounded between the weighted spaces
  \begin{equation*}
    T:L^{2}(\bra{x}^{2s}dx)\to 
      \ell^{2}_{\bra{k}^{2 \sigma}}(L^{2}(\bra{\lambda}^{2s}d \lambda)).
  \end{equation*}
  When $s=0$ we have by Plancherel's Theorem and by \eqref{eq:equiv1}
  \begin{equation*}
  \begin{split}
    \sum \bra{k}^{2 \sigma}
      \left\|\mathcal{F}_{\lambda\to y}\left(
          \one{+}(\lambda)\lambda^{\frac{n-1}{2}}
          \check f^{l}_{k}(\lambda)\right)\right\|^{2}_{L^{2}_{y}}
      \simeq &
    \\
    \simeq
    \sum \bra{k}^{2 \sigma}
      \left\|\lambda^{\frac{n-1}{2}}
          \check f^{l}_{k}(\lambda)\right\|
          ^{2}_{L^{2}_{\lambda}(0,\lambda)}
    &\simeq \|\Lambda^{\sigma}_{\omega}\check f\|
       ^{2}_{L^{2}(\mathbb{R}^{n})}.
  \end{split}
  \end{equation*}
  Since $\Lambda_{\omega}$
  commutes with the Fourier transform, indeed
  \begin{equation*}
    \mathcal{F}(-\Delta_{S}f)=
    \mathcal{F}\sum(x_{j}\partial_{k}-x_{j}\partial_{j})^{2}f
    =\sum(\partial_{j}\xi_{k}-\partial_{k}\xi_{j})^{2}
    \mathcal{F}f,
  \end{equation*}
  again by Plancherel we obtain \eqref{eq:genineq2} for $s=0$.
  
  To handle the case $s=1$ we consider the quantity
  \begin{equation*}
  \begin{split}
    \left\|y\mathcal{F}_{\lambda\to y}\left(
        \one{+}(\lambda)\lambda^{\frac{n-1}{2}}
        \check f^{l}_{k}(\lambda)\right)\right\|^{2}_{L^{2}_{y}}
    =  &
    \left\|\partial_{\lambda}\left(
        \one{+}(\lambda)\lambda^{\frac{n-1}{2}}
        \check f^{l}_{k}(\lambda)\right)\right\|^{2}_{L^{2}_{\lambda}}
    \lesssim 
    \\
    \lesssim
    \left\|\lambda^{\frac{n-1}{2}}
        \partial_{\lambda}\check f^{l}_{k}(\lambda)
        \right\|^{2}_{L^{2}_{\lambda}(0,\infty)}+  &
        \left\|\lambda^{\frac{n-3}{2}}
            \check f^{l}_{k}(\lambda)
            \right\|^{2}_{L^{2}_{\lambda}(0,\infty)}.
  \end{split}
  \end{equation*}
  Multiplying by $\bra{k}^{2 \sigma}$, summing over $l,k$ and
  recalling \eqref{eq:equiv3}, we obtain
  \begin{equation*}
    \sum \bra{k}^{2 \sigma}
      \left\|y\mathcal{F}_{\lambda\to y}\left(
          \one{+}(\lambda)\lambda^{\frac{n-1}{2}}
          \check f^{l}_{k}(\lambda)\right)\right\|^{2}_{L^{2}_{y}}
      \lesssim
    \|\nabla \Lambda^{\sigma}_{\omega}\check f\|
      ^{2}_{L^{2}(\mathbb{R}^{n})}
    + \left\|\lambda^{\frac{n-3}{2}}
          \check f^{0}_{0}(\lambda)
          \right\|^{2}_{L^{2}_{\lambda}(0,\infty)}
  \end{equation*}
  where the last term can not estimated by \eqref{eq:equiv3}
  because of the factor $k^{2}$ which vanishes when $k=0$.
  However we have
  \begin{equation*}
    \check f^{0}_{0}(\lambda)=\int_{|\omega|=1}\check f(\lambda \omega)
    d \lambda=
    \int_{|\omega|=1}\Lambda^{\sigma}_{\omega}\check f(\lambda \omega)
    d \lambda
  \end{equation*}
  which implies, using Hardy's inequlity
  \begin{equation*}
    \left\|\lambda^{\frac{n-3}{2}}
          \check f^{0}_{0}(\lambda)
          \right\|^{2}_{L^{2}_{\lambda}(0,\infty)}
    \lesssim 
    \left\|\frac{\Lambda^{\sigma}_{\omega}\check f}
                 {|\xi|}\right\|_{L^{2}(\mathbb{R}^{n})}
    \lesssim\|\nabla \Lambda^{\sigma}_{\omega}\check f\|_{L^{2}}.
  \end{equation*}
  Thus we have proved
  \begin{equation*}
    \sum \bra{k}^{2 \sigma}
      \left\|y\mathcal{F}_{\lambda\to y}\left(
          \one{+}(\lambda)\lambda^{\frac{n-1}{2}}
          \check f^{l}_{k}(\lambda)\right)\right\|^{2}_{L^{2}_{y}}
      \lesssim
      \|\nabla \Lambda^{\sigma}_{\omega}\check f\|
        ^{2}_{L^{2}(\mathbb{R}^{n})}
      \simeq\||x| \Lambda^{\sigma}_{\omega}f\|_{L^{2}}^{2}
  \end{equation*}
  again by the commutation of $\Lambda_{\omega}$ with the Fourier
  transform. This gives \eqref{eq:genineq2} for $s=1$ and
  concludes the proof of the Theorem.
\end{proof}

\begin{remark}\label{rem:angularfree}
  Since the operator $\Lambda_{\omega}$ commutes with $|D|$, 
  estimates \eqref{eq:strichartz1}, \eqref{eq:strich3D} and
  \eqref{eq:strichartz2} obviously generalize to the following;
  for any real $s\ge0$,
  \begin{equation}\label{eq:strichartzang1}
    \|\Lambda_{\omega}^{s} e^{it|D|}f\|
         _{L^2_tL^\infty_{r}L^2_\omega}
    \lesssim
    \|\Lambda^{s+ \sigma}_{\omega} f\|_{\dot H^\frac{n-1}{2}}
  \end{equation}
  and
  \begin{equation}\label{eq:strichartzang2}
    \left\|\Lambda_{\omega}^{s} 
      \int_0^t e^{i(t-s)|D|}F(s,x)ds
    \right\|_{L^2_tL^\infty_{|x|}L^2_\omega}
    \lesssim
    \|
      \bra{x}^{\frac{1}{2}+} 
      |D|^{\frac{n-1}{2}}\Lambda_{\omega}^{s+\sigma}F
    \|_{L^2_tL^2_{x}}
  \end{equation}
  where
  \begin{equation*}
    \sigma=1-\frac n2 \quad\text{if $n\ge4$},\qquad
    \sigma=0 \quad\text{if $n=3$}.
  \end{equation*}
\end{remark}

From the previous estimate for the free wave equation
it is not difficult to obtain analogous
endpoint Strichartz and Strichartz-smoothing estimates
for the 3D Dirac system:

\begin{corollary}\label{cor:freedirac}
  Let $n=3$.
  Then the flow $e^{it \mathcal{D}}$ satisfies,
  for all $s\ge0$, the estimates
  \begin{equation}\label{eq:freedirac}
    \| \Lambda_{\omega}^{s} 
         e^{it \mathcal{D}}f\|_{L^2_tL^\infty_{r}L^2_\omega}
    \lesssim
    \|\Lambda_{\omega}^{s} f\|_{\dot H^1},
  \end{equation}
  and
  \begin{equation}\label{eq:freediracnh}
    \left\|\Lambda_{\omega}^{s} 
      \int_0^t e^{i(t-t')\mathcal{D}}F(t',x)dt'
    \right\|_{L^2_tL^\infty_{|x|}L^2_\omega}
    \lesssim
    \|
      \bra{x}^{\frac{1}{2}+}|D|\Lambda_{\omega}^{s} F
    \|_{L^2_tL^2_{x}}.
  \end{equation}
\end{corollary}

\begin{proof}%[of ...]
  If $u$ solves the problem
  \begin{equation}\label{pbdir}
  iu_t+\mathcal{D}u=0,\qquad u(0)=f(x),
  \end{equation}
  by applying the operator $(i\partial_t-\mathcal{D})$, we see that
  $u$ solves also
  \begin{equation}\label{pbwave}
  \square u=0,\qquad
  u(0)=f(x),\qquad
  u_t(0)=i\mathcal{D}f.
  \end{equation}
  This gives the representation
  \begin{equation}\label{eq:reprdirac}
    e^{it \mathcal{D}}f=
    \cos(t|D|)f+ i \frac{\sin(t|D|)}{|D|}\mathcal{D}f.
  \end{equation}
  Moreover, we recall that the Riesz operators $|D|^{-1}\partial_{j}$
  are bounded on weighted $L^{2}$ spaces with weight $\bra{x}^{a}$
  for $a<n/2$. Thus in the case $s=0$ estimates
  \eqref{eq:freedirac}, \eqref{eq:freediracnh} are immediate consequences
  of the corresponding estimates for the wave equation proved above.
  
  In order to complete the proof in the case $s>0$,
  we need analyze the structure of the Dirac operator
  $\mathcal{D}$ in greater detail.
  Following \cite{Thaller92-a}, we know that the space
  $L^{2}(\mathbb{R}^{3})^{4}$ is isomorphic to an orthogonal direct sum
  \begin{equation*}
    L^{2}(\mathbb{R}^{3})^{4}\simeq
    \bigoplus_{j=\frac12,\frac32,\dots}^{\infty}
    \bigoplus_{m_{j}=-j}^{j}
    \bigoplus_{k_{j}=\atop\pm(j+1/2)}
    L^{2}(0,+\infty;dr)
    \otimes
    H_{m_{j},k_{j}}.
  \end{equation*}
  Each space $H_{m_{j},k_{j}}$ has dimension two and is
  generated by the orthonormal basis 
  $\{\Phi^{+}_{m_{j},k_{j}},\Phi^{-}_{m_{j},k_{j}}\}$,
  which can be explicitly written in terms of
  spherical harmonics: when $k_{j}=j+1/2$ we have
  \begin{equation*}
    \Phi^{+}_{m_{j},k_{j}}=
     \frac{i}{\sqrt{2j+2}}
    \begin{pmatrix}
      \sqrt{j+1-m_{j}}\ \ Y^{m_{j}-1/2}_{k_{j}} \\
      -\sqrt{j+1+m_{j}}\ \ Y^{m_{j}+1/2}_{k_{j}}\\
      0 \\
      0
    \end{pmatrix}
  \end{equation*}
  \begin{equation*}
    \Phi^{-}_{m_{j},k_{j}}=
     \frac{1}{\sqrt{2j}}
    \begin{pmatrix}
      \sqrt{j+m_{j}}\ \ Y^{m_{j}-1/2}_{k_{j}-1} \\
      \sqrt{j-m_{j}}\ \ Y^{m_{j}+1/2}_{k_{j}-1}\\
      0 \\
      0
    \end{pmatrix}
  \end{equation*}
  while when $k_{j}=-(j+1/2)$ we have
  \begin{equation*}
    \Phi^{+}_{m_{j},k_{j}}=
     \frac{i}{\sqrt{2j}}
    \begin{pmatrix}
      \sqrt{j+m_{j}}\ \ Y^{m_{j}-1/2}_{1-k_{j}} \\
      \sqrt{j-m_{j}}\ \ Y^{m_{j}+1/2}_{1-k_{j}}\\
      0 \\
      0
    \end{pmatrix}
  \end{equation*}
  \begin{equation*}
    \Phi^{-}_{m_{j},k_{j}}=
     \frac{1}{\sqrt{2j+2}}
    \begin{pmatrix}
      \sqrt{j+1-m_{j}}\ \ Y^{m_{j}-1/2}_{-k_{j}} \\
      -\sqrt{j+1+m_{j}}\ \ Y^{m_{j}+1/2}_{-k_{j}}\\
      0 \\
      0
    \end{pmatrix}.
  \end{equation*}
  The isomorphism is expressed by the explicit expansion
  \begin{equation}\label{eq:expans}
    \Psi(x)=\sum 
      \frac1r
      \psi^{+}_{m_{j},k_{j}}(r)
      \Phi^{+}_{m_{j},k_{j}}+
      \frac1r
      \psi^{-}_{m_{j},k_{j}}(r)
      \Phi^{-}_{m_{j},k_{j}}
  \end{equation}
  with
  \begin{equation}\label{eq:L2exp}
    \|\Psi\|^{2}_{L^{2}}=
    \sum\int_{0}^{\infty}[
      |\psi^{+}_{m_{j},k_{j}}|^{2}+|\psi^{-}_{m_{j},k_{j}}|^{2}
    ]dr.
  \end{equation}
  Notice also that
  \begin{equation}\label{eq:L2expS}
    \|\Psi\|^{2}_{L^{2}_{\omega}}=
      \sum
      \frac{1}{r^{2}}
      |\psi^{+}_{m_{j},k_{j}}|^{2}+
      \frac{1}{r^{2}}|
      \psi^{-}_{m_{j},k_{j}}|^{2}.
  \end{equation}
  Each $L^{2}(0,+\infty; dr)\otimes H_{m_{j},k_{j}}$ 
  is an eigenspace of the
  Dirac operator $\mathcal{D}=i^{-1}\sum \alpha_{j}\partial_{j}$
  and the action of $\mathcal{D}$ can be written, 
  in terms of the expansion \eqref{eq:expans}, as
  \begin{equation*}
    \mathcal{D}\Psi=\sum 
      \left(-\frac{d}{dr}\psi^{-}_{m_{j},k_{j}}+
        \frac{k_{j}}{r}\psi^{-}_{m_{j},k_{j}}
      \right)
      \frac{\Phi^{+}_{m_{j},k_{j}}}{r}+
      \left(\frac{d}{dr}\psi^{+}_{m_{j},k_{j}}+
          \frac{k_{j}}{r}\psi^{+}_{m_{j},k_{j}}
      \right)
      \frac{\Phi^{-}_{m_{j},k_{j}}}{r}.
  \end{equation*}
  From decomposition \eqref{eq:expans} it is clear that the operator 
  $\Lambda_{\omega}^{\sigma}$, which acts on spherical harmonics as
  \begin{equation}\label{eq:actLa}
    \Lambda_{\omega}^{\sigma}Y^{m}_{\ell}=
    (1+\ell(\ell+1))^{\frac{\sigma}{2}}\cdot Y^{m}_{\ell},
  \end{equation}
  does not commute with $\mathcal{D}$. Indeed, each space
  $H_{m_{j},k_{j}}$ involves two spherical harmonics $Y^{m}_{\ell}$
  with two values of $\ell$ which differ by 1, and $\mathcal{D}$
  swaps them. However, the modified operator 
  $\widetilde\Lambda_{\omega}^{\sigma}$ defined by
  \begin{equation}\label{eq:actLati}
    \widetilde\Lambda_{\omega}^{\sigma}
    \Phi^{\pm}_{m_{j},k_{j}}=
    |k_{j}|^{\sigma}\Phi^{\pm}_{m_{j},k_{j}}
  \end{equation}
  obviously commutes with $\mathcal{D}$, 
  thus estimates \eqref{eq:freedirac}, 
  \eqref{eq:freediracnh} are trivially true if we replace
  $\Lambda$ with $\widetilde{\Lambda}$. It remains to show that
  we obtain equivalent norms. The equivalence
  \begin{equation*}
    \|\widetilde\Lambda_{\omega}^{\sigma}f\|_{L^{2}_{\omega}}
    \simeq
    \|\Lambda_{\omega}^{\sigma}f\|_{L^{2}_{\omega}}
  \end{equation*}
  follows directly from \eqref{eq:actLa}, \eqref{eq:actLati} and
  \eqref{eq:L2expS}. 
  Moreover, $\widetilde{\Lambda}$
  and $\Lambda$ commute with $\Delta$, hence with $|D|$, and this
  implies
  \begin{equation*}
    \| |D|\Lambda_{\omega}^{s} f\|_{L^{2}}
    \simeq
    \||D|\widetilde 
        \Lambda_{\omega}^{s} f\|_{L^{2}}
  \end{equation*}
  or, equivalently,
  \begin{equation*}
    \|\Lambda_{\omega}^{s} f\|_{\dot H^{1}}
    \simeq
    \|\widetilde 
        \Lambda_{\omega}^{s} f\|_{\dot H^{1}}.
  \end{equation*}
  This is sufficient to prove \eqref{eq:freedirac}.
  Since $\widetilde{\Lambda}$ and $\Lambda$ also commute with
  radial weights we have
  \begin{equation*}
    \| \bra{x}^{\frac{1}{2}+}|D|\Lambda_{\omega}^{s} f\|_{L^{2}}
    \simeq
    \| \bra{x}^{\frac{1}{2}+}|D|\widetilde 
        \Lambda_{\omega}^{s} f\|_{L^{2}}
  \end{equation*}
  which gives \eqref{eq:freediracnh}.
\end{proof}

% section estimates (end)

\section{The wave equations with potential}\label{sec:wWE_pot}  %(fold)

Our next goal is to extend the results of previous section 
to the case of perturbed flows. This will be obtain1ed by a
perturbative argument, relying on the smoothing estimates
of \cite{DanconaFanelli08-a} and the mixed Strichartz-smoothing
estimates of the previous section.
In \cite{DanconaFanelli08-a} smoothing
estimates were proved for several classes of dispersive equations
perturbed with electromagnetic potentials
(while the 1D case was analyzed in \cite{DanconaFanelli06-a}).
For the wave equation in dimension $n\ge3$
the estimates are the following:

\begin{proposition}\label{pro:smooWE}
  Let $n\geq3$. Assume the operator
  $$-\Delta+W(x,D)=-\Delta+a(x)\cdot \nabla+b_{1}(x)+b_{2}(x)$$
  is selfadjoint and its coefficients satisfy
  \begin{equation}\label{eq:ipa}
    |a(x)|\le \frac{\delta}
       {|x|^{1-\epsilon}+|x|^{2}|\log|x||^{\sigma}}
  \end{equation}
  \begin{equation}\label{eq:ipb}
    |b_{1}(x)|\le \frac{\delta}{|x|^{1-\epsilon}+|x|^{2}},\qquad
    0\le b_{2}(x)\le \frac{C}{|x|^{1-\epsilon}+|x|^{2}}
  \end{equation}
  for some $\delta,\epsilon>0$ sufficiently small and
  some $\sigma>1/2$, $C>0$. Moreover assume that
  0 is not a resonance for $-\Delta+b_{2}$.
  Then the following smoothing estimate holds:
   \begin{equation}\label{eq:smooWE}
     \|
     (|x|^{\frac12-\epsilon}+|x|)^{-1}
     e^{it\sqrt{-\Delta+W}}f\|_{L^2L^2}
   \lesssim\|f\|_{L^2}.
  \end{equation}
\end{proposition}

The assumption that 0 is not a resonance for $-\Delta+b_{2}(x)$ here
means: if $(-\Delta+b_{2})f=0$ and $\bra{x}^{-1}f\in L^{2}$ then $f
\equiv0$.

Combining Proposition \ref{pro:smooWE} with \eqref{eq:strichartz2} we 
obtain the following Strichartz
endpoint estimate for the 3D wave equation perturbed with
an electric potential:

\begin{theorem}\label{the:strichWE}
  Let $n=3$ and consider the Cauchy problem for the wave equation
  $$u_{tt}-\Delta u+V(x)u=F,\qquad
  u(0,x)=f(x),\qquad u_{t}(0,x)=g(x)$$
  under the assumptions:
  \begin{enumerate}\setlength{\itemindent}{-10pt}
    \renewcommand{\labelenumi}{(\roman{enumi})}
    \item $V(x)$ is real valued and
    the positive and negative parts $V_{\pm}$ satisfy
    \begin{equation}\label{eq:assV}
      V_{+}\le \frac{C}{|x|^{\frac12-\epsilon}+|x|^{2}},\qquad
      V_{-}\le \frac{\delta}{|x|^{\frac12-\epsilon}+|x|^{2}}
    \end{equation}
    for some $\delta,\epsilon$ sufficiently small and some $C\ge0$;
    \item $-\Delta+V$ is selfadjoint;
    \item 0 is not a resonance for $-\Delta+V_{-}$.
  \end{enumerate}
  Then the solution $u(t,x)$
  satisfies the endpoint Strichartz estimate
  \begin{equation}\label{eq:endWEV}
    \| u\|_{L^2_tL^\infty_{r}L^2_\omega}
    \lesssim
    \|f\|_{\dot H^{1}}
    +
    \|g\|_{L^{2}}
    +
    \|\bra{x}^{\frac12+}
      F\|_{L^{2}_{t}L^{2}_{x}}.
  \end{equation}
\end{theorem}

\begin{proof}%[of ...]
  We represent $u(t,x)$ in the form
  \begin{equation*}
    u(t,x)=I+II-III
  \end{equation*}
  where
  \begin{equation*}
    I=\cos(t|D|)f+\frac{\sin(t|D|)}{|D|}g,
  \end{equation*}
  \begin{equation*}
    II= \int_{0}^{t}|D|^{-1}\sin((t-s)|D|)F \,ds,
  \end{equation*}
  and
  \begin{equation*}
    III= \int_{0}^{t}|D|^{-1}\sin((t-s)|D|)Vu \,ds.
  \end{equation*}
  We can use \eqref{eq:strichartz1} to estimate $I$ 
  and \eqref{eq:strichartz2} to estimate $II$ 
  in the norm $L^{2}_{t}L^{\infty}_{r}L^{2}_{\omega}$
  directly. On the other
  hand, applying \eqref{eq:strichartz2} to $III$ we get
  \begin{equation*}
    \|III\|
    _{L^{2}_{t}L^{\infty}_{r}L^{2}_{\omega}}
    \lesssim
    \|\bra{x}^{\frac12+}Vu\|_{L^{2}L^{2}}\le 
    \|\bra{x}^{\frac12+}\tau_{\epsilon}V\|_{L^{\infty}_{x}}
    \|\tau_{\epsilon}^{-1}u\|_{L^{2}L^{2}}.
  \end{equation*}
  By assumption $\bra{x}^{\frac12+}\tau_{\epsilon}V$ is bounded on
  $\mathbb{R}^{n}$, moreover we are allowed to use
  \eqref{eq:smooWE} since $V$ satisfies the assumptions of
  Proposition \ref{pro:smooWE}. Notice that \eqref{eq:smooWE}
  implies
  \begin{equation*}
    \|\tau_{\epsilon}^{-1}u\|_{L^{2}L^{2}}\lesssim
    \|f\|_{L^{2}}+\||D|^{-1}g\|_{L^{2}}
  \end{equation*}
  and in conclusion we have proved
  \begin{equation*}
    \|III\|
    _{L^{2}_{t}L^{\infty}_{r}L^{2}_{\omega}}
    \lesssim
    \|f\|_{L^{2}}+\||D|^{-1}g\|_{L^{2}}
  \end{equation*}
  which completes the proof of \eqref{eq:endWEV}.
\end{proof}

Analogous estimates can be proved for the Klein-Gordon equation,
or in higher dimension $n\ge3$,
for first order perturbations, and for angular derivatives of
the solutions. We omit the details since we prefer to focus
on the Dirac equation here.

% section wave_equations_with_potentials (end)

\section{The Dirac equation with potential}\label{sec:dirac_potential}  %(fold)

We consider now the perturbed Dirac operator $\mathcal{D}+V(x)$
where $V(x)$ is a small $4 \times 4$ hermitian matrix valued
potential. We prove here more general versions of the estimates
given in \cite{DanconaFanelli08-a}, \cite{DanconaFanelli07-a}
in order to include angular regularity. 
We begin with the free Dirac equation:

\begin{proposition}\label{pro:smoofreeD}
  The free Dirac flow satisfies, for all $\sigma>1$ 
  and $s\ge0$, the smoothing estimates (with
  $w_{\sigma}(x)=|x|(1+|\log|x||)^\sigma$)
  \begin{equation}\label{eq:smoofreeD}
    \|w_\sigma^{-1/2}\Lambda^{s}_{\omega}
       e^{it\mathcal{D}}f\|_{L^2_tL^2_x}\lesssim
    \| \Lambda^{s}_{\omega}f\|_{L^2}
  \end{equation}
  and
  \begin{equation}\label{eq:smoofreeDnh}
    \left\|w_\sigma^{-1/2}\Lambda^{s}_{\omega}\int_{0}^{t}
      e^{i(t-t')\mathcal{D}}F(t')dt'\right\|_{L^2_tL^2_x}\lesssim
    \| w_\sigma^{1/2}\Lambda^{s}_{\omega}F\|_{L^2_{t}L^{2}_{x}}
  \end{equation}
\end{proposition}

\begin{proof}%[of ...]
  When $s=0$, both estimates follow from the resolvent estimate
  \begin{equation*}
    \|w_\sigma^{-1/2}R_{\mathcal{D}}(z)f\|_{L^{2}(\mathbb{R}^{3})}
    \le C\|w_\sigma^{1/2}f\|_{L^{2}(\mathbb{R}^{3})},\qquad
    z\not\in \mathbb{R},
  \end{equation*}
  with a constant uniform in $z$, proved in
  \cite{DanconaFanelli07-a}
  using a standard application of Kato's
  theory (see also \cite{DanconaFanelli08-a}). The case $s>0$
  is proved exactly as in Corollary \ref{cor:freedirac}, first
  by replacing $\Lambda_{\omega}$ with $\widetilde{\Lambda}_{\omega}$
  which commutes with the flow,
  and then by using the equivalence of norms.
\end{proof}

We consider now the case of a perturbed Dirac system
\begin{equation*}
  iu_{t}=\mathcal{D}u+Vu
\end{equation*}
where $V(x)$ is a $4 \times 4$ matrix potential. If $V$ is
hermitian and its weak $L^{3,\infty}$ norm is small enough, the
operator $\mathcal{D}+V$ is selfadjoint as proved in
\cite{DanconaFanelli07-a}. In all of the following results the
assumptions on the potential are somewhat stronger than this,
so in all cases the unitary flow 
$e^{it(\mathcal{D}+V)}$ will be well defined and continuous on
$L^{2}(\mathbb{R}^{3})^{4}$ by spectral theory.

\begin{proposition}\label{pro:smooD}
  Let $V(x)$ be a hermitian $4\times4$ matrix on
  $\mathbb{R}^{3}$ such that
  \begin{equation}\label{eq:Vhp}
    |V(x)|\leq\frac{\delta}{w_\sigma(|x|)},\qquad
    w_{\sigma}(r)=r \cdot(1+|\log r|)^\sigma
  \end{equation}
  for some $\delta>0$ sufficiently small and some $\sigma>1$.
  Then the perturbed Dirac flow $e^{it(\mathcal{D}+V)}$
  satisfies the smoothing estimates
  \begin{equation}\label{eq:smoothdir}
    \|w_\sigma^{-1/2}e^{it(\mathcal{D}+V)}f\|_{L^2_tL^2_x}\lesssim
    \| f\|_{L^2},
  \end{equation}
  \begin{equation}\label{eq:smoothdirnh}
    \left\|w_\sigma^{-1/2}\int_{0}^{t}
        e^{i(t-t')(\mathcal{D}+V)}F(t')dt'\right\|_{L^2_tL^2_x}\lesssim
    \| w_\sigma^{1/2}F\|_{L^{2}_{t}L^2_{x}}.
  \end{equation}
  % If in addition $V=V(|x|)$ is a radial potential,
  % we have for all $s\ge0$ the estimates with angular
  % regularity
%%%%%%%%%%%%%%%%%%  
  If in addition $V$ satisfies for some $s>1$
  the condition
  \begin{equation}\label{eq:angV}
    \|\Lambda_{\omega}^{s}
         V(r\ \cdot\ )\|_{L^{2}(\mathbb{S}^{2})}
    \le \frac{\delta}{w_{\sigma}(r)},
  \end{equation}
  then we have, for all $0\le s\le 2$, the estimates with angular
  regularity
%%%%%%%%%%%%%%%%%%  
  \begin{equation}\label{eq:smoothdirang}
    \|w_\sigma^{-1/2}\Lambda^{s}_{\omega}
       e^{it(\mathcal{D}+V)}f\|_{L^2_tL^2_x}\lesssim
    \| \Lambda^{s}_{\omega}f\|_{L^2},
  \end{equation}
  \begin{equation}\label{eq:smoothdirangnh}
    \left\|w_\sigma^{-1/2}\Lambda^{s}_{\omega}\int_{0}^{t}
        e^{i(t-t')(\mathcal{D}+V)}F(t')dt'\right\|_{L^2_tL^2_x}\lesssim
    \| w_\sigma^{1/2}\Lambda^{s}_{\omega}F\|_{L^{2}_{t}L^2_{x}}.
  \end{equation}
\end{proposition}

\begin{proof}%[of ...]
  If $u$ solves
  \begin{equation*}
    iu_{t}=\mathcal{D}u+Vu+F,\qquad u(0)=f
  \end{equation*}
  we can write
  \begin{equation*}
    u=
    e^{it\mathcal{D}}f+i
    \int_{0}^{t}e^{i(t-t')\mathcal{D}}[Vu(t')+F(t')]dt'.
  \end{equation*}
  Using \eqref{eq:smoofreeD}, \eqref{eq:smoofreeDnh} with $s=0$
  and assumption \eqref{eq:Vhp} we get
  \begin{equation*}
  \begin{split}
    \|w_\sigma^{-1/2}u\|_{L^2_tL^2_x}
    &\lesssim
    \| f\|_{L^2}+\|w_{\sigma}^{1/2}[Vu+F]\|_{L^{2}L^{2}}\le 
    \\
    &\le\| f\|_{L^2}+\delta \|w_{\sigma}^{-1/2}u\|_{L^{2}L^{2}}
    +\|w_{\sigma}^{1/2}F\|_{L^{2}L^{2}}.
  \end{split}
  \end{equation*}
  If $\delta$ is sufficiently small this implies 
  both \eqref{eq:smoothdir} and \eqref{eq:smoothdirnh}.
  
  To prove \eqref{eq:smoothdirang}, \eqref{eq:smoothdirangnh}
  we proceed in a similar way using again 
  \eqref{eq:smoofreeD} and \eqref{eq:smoofreeDnh}:
  \begin{equation*}
    \|w_\sigma^{-1/2}\Lambda^{s}_{\omega}
      u\|_{L^2_tL^2_x}\lesssim
    \| \Lambda^{s}_{\omega}f\|_{L^2}+
    \|w_{\sigma}^{1/2}\Lambda^{s}_{\omega}(Vu+F)\|_{L^{2}L^{2}}.
  \end{equation*}
  We shall need the following fairly elementary product
  estimate involving the
  angular derivative operator $\Lambda_{\omega}$
  \begin{equation}\label{eq:prodest}
    \|\Lambda^{s}_{\omega}(gh)\|_{L^{2}_{\omega}(\mathbb{S}^{2})}
    \lesssim
    \|\Lambda^{s}_{\omega}g\|_{L^{2}_{\omega}(\mathbb{S}^{2})}
    \|\Lambda^{s}_{\omega}h\|_{L^{2}_{\omega}(\mathbb{S}^{2})}
  \end{equation}
  which holds provided $s>1$. This estimate can be proved 
  e.g.~by localizing the norm on the sphere via a finite partition
  of unity, and then applying in each coordinate patch
  a standard product estimate
  in the Sobolev space $H^{s}(\mathbb{R}^{2})$, $s>1$.
  
  Applying \eqref{eq:prodest}, and using assumption
  \eqref{eq:angV}, we have
  \begin{equation*}
    \|w_{\sigma}^{1/2}\Lambda^{s}_{\omega}(Vu+F)\|_{L^{2}L^{2}}\le
    \delta\|w_{\sigma}^{-1/2}\Lambda^{s}_{\omega}u\|_{L^{2}L^{2}}+
    \|w_{\sigma}^{1/2}\Lambda^{s}_{\omega}F\|_{L^{2}L^{2}}
  \end{equation*}
  and the proof is concluded as above.
\end{proof}

We note the following consequence of \eqref{eq:smoothdir}:

\begin{corollary}\label{cor:nablau}
  Assume that the hermitian matrix $V(x)$ satisfies,
  for $\delta$ sufficiently small, $C$ arbitrary and $\sigma>1$
  (with $w_{\sigma}(r)=r(1+|\log r|)^\sigma$)
  \begin{equation}\label{eq:assnablaV}
    |V(x)|\leq\frac{\delta}{w_\sigma(|x|)},\qquad
    |\nabla V(x)|\leq\frac{C}{w_\sigma(|x|)}.
  \end{equation}
  Then besides \eqref{eq:smoothdir} we have the
  estimate for the derivatives of the flow
  \begin{equation}\label{eq:smoonablau}
    \|w_\sigma^{-1/2}
       \nabla
       e^{it(\mathcal{D}+V)}f\|_{L^2_tL^2_x}\lesssim
    \| f\|_{H^{1}}.
  \end{equation}
  % If in addition we assume that $V=V(|x|)$ is a radial
  % potential, we have for all $s\ge0$
  % the estimate with angular regularity
  If in addition we assume that, for some $s>1$,
  \begin{equation}\label{eq:nablaangV}
    \|\Lambda_{\omega}^{s} V(r \cdot)\|_{L^{2}(\mathbb{S}^{2})}
    \le \frac{\delta}{w_{\sigma}(r)},
    \qquad
    \|\Lambda_{\omega}^{s}\nabla V(r \cdot)\|_{L^{2}(\mathbb{S}^{2})}
    \le \frac{C}{w_{\sigma}(r)},
  \end{equation}
  then we have the following estimate with angular regularity
  \begin{equation}\label{eq:smoonablauang}
    \|w_\sigma^{-1/2}
       \nabla \Lambda^{s}_{\omega}
       e^{it(\mathcal{D}+V)}f\|_{L^2_tL^2_x}\lesssim
    \| \Lambda^{s}_{\omega}f\|_{H^{1}}.
  \end{equation}
  
\end{corollary}

\begin{proof}%[of ...]
  Assume at first $s=0$ and let $u=e^{it(\mathcal{D}+V)}f$.
  Each derivative $u_{j}=\partial_{j}u$ satisfies an equation like
  \begin{equation*}
    i\partial_{t}u_{j}=\mathcal{D}u_{j}+Vu_{j}+V_{j}u,\qquad
    V_{j}=\partial_{j}V,\qquad
    u_{j}(0,x)=f_{j}=\partial_{j}f
  \end{equation*}
  so we can represent it in the form
  \begin{equation*}
    u_{j}=e^{it(\mathcal{D}+V)}f_{j}+
      i \int_{0}^{t}e^{i(t-s)(\mathcal{D}+V)}V_{j}uds.
  \end{equation*}
  To the first term at the r.h.s.~we can apply estimate 
  \eqref{eq:smoothdir} obtaining
  \begin{equation*}
    \|w_\sigma^{-1/2}e^{it(\mathcal{D}+V)}f_{j}\|_{L^2_tL^2_x}
    \lesssim\|f\|_{\dot H^{1}}. 
  \end{equation*}
  To handle the second term we use \eqref{eq:smoothdirnh}:
  \begin{equation*}
     \left\|w_\sigma^{-1/2}
        \int e^{i(t-s)(\mathcal{D}+V)}V_{j}u
      \right\|_{L^{2}}\lesssim
      \|w_\sigma^{1/2}V_{j}u\|_{L^2_tL^2_x}\le 
      \|w_{\sigma}V_{j}\|_{L^{\infty}_{x}}
      \|w_\sigma^{-1/2}u\|_{L^2_tL^2_x}
  \end{equation*}
  and again by \eqref{eq:smoothdir} and by the assumption on
  $\nabla V$ we conclude the proof of \eqref{eq:smoonablau}.
  
  For the proof of \eqref{eq:smoonablauang} we apply to the
  equation for $u$ the operator $|D|$
  which commutes with $\mathcal{D}$:
  \begin{equation*}
    i \partial_{t}(|D|u)=
    \mathcal{D}(|D|u)+
    |D|(Vu)
  \end{equation*}
  and we use estimates \eqref{eq:smoothdirang}, 
  \eqref{eq:smoothdirangnh}, obtaining
  \begin{equation*}
    \|w_{\sigma}^{-1/2}|D|\Lambda^{s}_{w}u\|_{L^{2}L^{2}}\lesssim
    \|\Lambda^{s}_{w}f\|_{\dot H^{1}}+
    \|w_{\sigma}^{1/2}|D|\Lambda^{s}_{w}(Vu)\|_{L^{2}L^{2}}.
  \end{equation*}
  Now in the last term we commute $|D|$ with $\Lambda$ and
  we notice that we can replace $|D|$ by $\nabla$ obtaining
  an equivalent norm. This gives
  \begin{equation*}
    \|w_{\sigma}^{1/2}|D|\Lambda^{s}_{w}(Vu)\|_{L^{2}L^{2}}\le 
    \|w_{\sigma}^{1/2}\Lambda^{s}_{w}(\nabla V)u)\|_{L^{2}L^{2}}+
    \|w_{\sigma}^{1/2}\Lambda^{s}_{w}V(\nabla u)\|_{L^{2}L^{2}}.
  \end{equation*}
  We can now apply the product estimate \eqref{eq:prodest} and 
  assumptions \eqref{eq:nablaangV}; proceeding as in the
  first part of the proof we finally obtain \eqref{eq:smoonablauang}.
\end{proof}

By a similar perturbative argument, 
we obtain the endpoint Strichartz estimates for the Dirac
equation with potential. Notice that in the version of this
Theorem given in the Introduction
(Theorem \ref{the:strichD-i}) we used an equivalent formulation
in terms of the potential
$v_{\sigma}(x)=|x|^{\frac12}|\log|x||^{\sigma}+\bra{x}^{1+\sigma}$.

\begin{theorem}\label{the:strichD}
  Assume that the hermitian matrix $V(x)$ satisfies,
  for $\delta$ sufficiently small, $C$ arbitrary and $\sigma>1$
  (with $w_{\sigma}(r)=r(1+|\log r|)^\sigma$)
  \begin{equation}\label{eq:assnablaV2}
    |V(x)|\leq\frac{\delta}
         {\bra{x}^{\frac12+}w_{\sigma}(|x|)^{\frac12}},\qquad
    |\nabla V(x)|\leq\frac{C}
         {\bra{x}^{\frac12+}w_{\sigma}(|x|)^{\frac12}}.    
  \end{equation}
  Then the perturbed Dirac flow
  satisfies the endpoint Strichartz estimate
  \begin{equation}\label{eq:enddiracV}
    \|e^{it(\mathcal{D}+V)}f\|_{L^2_t L^\infty_{r}L^2_\omega}
    \lesssim
    \| f\|_{H^1}.
  \end{equation}
  % If in addition $V=V(|x|)$ is a radial potential,
  % we have for all $s\ge0$ the estimate with
  % angular regularity
  If instead we make the following
  assumption (which implies \eqref{eq:assnablaV2}):
  for some $s>1$,
  \begin{equation}\label{eq:nablaangV2}
    \|\Lambda_{\omega}^{s} V(r \cdot)\|_{L^{2}(\mathbb{S}^{2})}
    \le \frac{\delta}{\bra{r}^{\frac12+}w_{\sigma}(r)^{\frac12}},
    \qquad
    \|\Lambda_{\omega}^{s}\nabla V(r \cdot)\|_{L^{2}(\mathbb{S}^{2})}
    \le \frac{C}{\bra{r}^{\frac12+}w_{\sigma}(r)^{\frac12}},
  \end{equation}
  then we have the endpoint estimate with angular regularity
  \begin{equation}\label{eq:enddiracVang}
    \|\Lambda^{s}_{\omega}
      e^{it(\mathcal{D}+V)}f\|_{L^2_t L^\infty_{r}L^2_\omega}
    \lesssim
    \|\Lambda^{s}_{\omega} f\|_{H^1}
  \end{equation}
  and the energy estimate with angular regularity
  \begin{equation}\label{eq:energyang}
    \|\Lambda^{s}_{\omega}
      e^{it(\mathcal{D}+V)}f\|_{L^\infty_t H^{1}}
    \lesssim
    \|\Lambda^{s}_{\omega} f\|_{H^1}
  \end{equation}
\end{theorem}

\begin{proof}
  Consider first \eqref{eq:enddiracV}.
  Notice that $V$ satisfies in particular the assumptions of
  Corollary \ref{cor:nablau}. We can write
  \begin{equation}\label{eq:represI-II}
    e^{it(\mathcal{D}+V)}f=I+II
  \end{equation}
  with
  \begin{equation*}
    I=e^{it\mathcal{D}}f,\qquad
    II=i\int_{0}^{t}e^{i(t-s)\mathcal{D}}Vu\,ds.
  \end{equation*}
  The term $I$ is estimated directly using 
  \eqref{eq:freedirac} with $s=0$.
  On the other hand, applying \eqref{eq:freediracnh}
  to the term $II$ we get
  \begin{equation*}
    \|II\|_{L^2_tL^\infty_{r}L^2_\omega}
    \lesssim
    \|\bra{x}^{\frac12+}|D|(Vu)\|_{L^{2}L^{2}}.
    % \le 
    % \|\bra{x}^{\frac12+}w_{\sigma}V\|_{L^{\infty}_{x}}
    % \|w_{\sigma}^{-1}u\|_{L^{2}L^{2}}
  \end{equation*}
  Now we recall that the Riesz operators $|D|^{-1}\nabla$
  are bounded on weighted $L^{2}$ spaces with $A_{2}$ weights,
  and $\bra{x}^{s}$ belongs to this class provided $s<n/2$
  (see \cite{Stein93-a}). Thus we can continue the chain of
  inequalities as follows:
  \begin{equation*}
    =\|\bra{x}^{\frac12+}|D|^{-1}\nabla |D|(Vu)\|_{L^{2}L^{2}}
    \lesssim
    \|\bra{x}^{\frac12+}\nabla(Vu)\|_{L^{2}L^{2}}
    \lesssim A+B
  \end{equation*}
  where
  \begin{equation*}
    A=\|\bra{x}^{\frac12+}(\nabla V)u\|_{L^{2}L^{2}},\qquad
    B=\|\bra{x}^{\frac12+}V\nabla u \|_{L^{2}L^{2}}
  \end{equation*}
  Then we have
  \begin{equation*}
    A\le 
    \|\bra{x}^{\frac12+}w_{\sigma}^{\frac12}\nabla V\|_{L^{\infty}}
    \|w_{\sigma}^{-\frac12}u\|_{L^{2}L^{2}}
    \lesssim \|f\|_{L^{2}}
  \end{equation*}
  by the assumptions on $\nabla V$ and \eqref{eq:smoothdir},
  while
  \begin{equation*}
    B\le 
    \|\bra{x}^{\frac12+}w_{\sigma}^{\frac12}V\|_{L^{\infty}}
    \|w_{\sigma}^{-\frac12}\nabla u\|_{L^{2}L^{2}}
    \lesssim \|f\|_{H^{1}}
  \end{equation*}
  by \eqref{eq:smoonablau}. Summing up, we arrive at
  \eqref{eq:enddiracV}.

  The proof of \eqref{eq:enddiracVang} is similar. We estimate
  $I$ using \eqref{eq:freedirac}. Applying \eqref{eq:freediracnh}
  to the term $II$ we get
  \begin{equation*}
    \|\Lambda^{s}_{\omega} II\|_{L^2_tL^\infty_{r}L^2_\omega}
    \lesssim
    \|\bra{x}^{\frac12+}|D|\Lambda^{s}_{\omega}(Vu)\|_{L^{2}L^{2}}.
  \end{equation*}
  Then we commute $|D|$ with $\Lambda_{\omega}$,
  and we can replace the operator $|D|$ with $\nabla$ since
  the norm is equivalent; we arrive at
  \begin{equation*}
    \|\Lambda^{s}_{\omega} II\|_{L^2_tL^\infty_{r}L^2_\omega}
    \lesssim 
    \|\bra{x}^{\frac12+}\Lambda^{s}_{\omega}(\nabla V)u\|_{L^{2}L^{2}}+
    \|\bra{x}^{\frac12+}\Lambda^{s}_{\omega}V(\nabla u)\|_{L^{2}L^{2}}.
  \end{equation*}
  Now we use the product estimate \eqref{eq:prodest} and
  assumptions \eqref{eq:nablaangV2} to obtain
  \begin{equation*}
    \lesssim
    C\|w_{\sigma}^{-1/2}\Lambda^{s}_{\omega}u\|_{L^{2}L^{2}}+
    \delta\|w_{\sigma}^{-1/2}\Lambda^{s}_{\omega}
       \nabla u\|_{L^{2}L^{2}}
  \end{equation*}
  and recalling the smoothing estimates 
  \eqref{eq:smoothdirang}, \eqref{eq:smoonablauang} we conclude
  the proof of \eqref{eq:enddiracVang}.
  
  It remains to prove \eqref{eq:energyang}. Consider first
  the free case $V \equiv0$. We have the conservation laws
  \begin{equation}\label{eq:consL2}
    \|e^{it \mathcal{D}}f\|_{L^{\infty}L^{2}}\equiv\|f\|_{L^{2}},
    \qquad
    \|\mathcal{D}
      e^{it \mathcal{D}}f\|_{L^{\infty}L^{2}}
    \equiv\|\mathcal{D}f\|_{L^{2}}
  \end{equation}
  which imply
  \begin{equation*}
    \|e^{it \mathcal{D}}f\|_{L^{\infty}H^{1}}
    \simeq\|f\|_{H^{1}}
  \end{equation*}
  since $\|\mathcal{D}f\|_{L^{2}}\simeq\|f\|_{\dot H^{1}}$.
  Moreover, the operator $\widetilde{\Lambda}_{\omega}$
  introduced in \eqref{eq:actLati} commutes with $\mathcal{D}$,
  so that we have for all $s\ge0$
  \begin{equation*}
    \|\widetilde{\Lambda}_{\omega}^{s}
       e^{it \mathcal{D}}f\|_{L^{\infty}H^{1}}
    \equiv\|\widetilde{\Lambda}_{\omega}f\|_{H^{1}}
  \end{equation*}
  and switching back to the equivalent operator $\Lambda_{\omega}$
  as in the proof of Corollary \ref{cor:freedirac} we obtain
  \begin{equation}\label{eq:energyangfree}
    \|{\Lambda}_{\omega}^{s}
       e^{it \mathcal{D}}f\|_{L^{\infty}H^{1}}
    \lesssim\|{\Lambda}_{\omega}f\|_{H^{1}}.
  \end{equation}
  Consider now the
  case $V\not \equiv 0$. We start from
  \begin{equation}\label{eq:IIfree}
    \|\bra{x}^{-\frac12-}
      e^{it \mathcal{D}}f\|_{L^{2}L^{2}}\le\|f\|_{L^{2}}
  \end{equation}
  which is a consequence of \eqref{eq:smoofreeD} (we relaxed
  the weight). Taking the dual of \eqref{eq:IIfree} we get
  \begin{equation*}
    \left\|
      \int e^{-it' \mathcal{D}}F(t')dt'
    \right\|_{L^{2}}\lesssim
    \|\bra{x}^{\frac12+}F\|_{L^{2}L^{2}}
  \end{equation*}
  which together with \eqref{eq:consL2} gives
  \begin{equation*}
    \left\|
      \int e^{i(t-t') \mathcal{D}}F(t')dt'
    \right\|_{L^{\infty} L^{2}}\lesssim
    \|\bra{x}^{\frac12+}F\|_{L^{2}L^{2}}.
  \end{equation*}
  Now a standard application of Christ-Kiselev' Lemma in the
  spirit of \cite{KeelTao98-a} (see also \cite{DanconaFanelli08-a} 
  for the case of Dirac equations) allows to replace the time 
  integral with a truncated integral and we obtain
  \begin{equation}\label{eq:mixedest}
    \left\|
      \int_{0}^{t} e^{i(t-t') \mathcal{D}}F(t')dt'
    \right\|_{L^{\infty} L^{2}}\lesssim
    \|\bra{x}^{\frac12+}F\|_{L^{2}L^{2}}.
  \end{equation}
  Recalling that the operator $\widetilde{\Lambda}_{\omega}$
  introduced in \eqref{eq:actLati} commutes with $\mathcal{D}$,
  and proceeding as in the proof of Corollary \ref{cor:freedirac}
  we obtain for all $s\ge0$
  \begin{equation}\label{eq:mixedestang}
    \left\|\Lambda_{\omega}^{s}
      \int_{0}^{t} e^{i(t-t') \mathcal{D}}F(t')dt'
    \right\|_{L^{\infty} L^{2}}\lesssim
    \|\bra{x}^{\frac12+}
       \Lambda_{\omega}^{s}F\|_{L^{2}L^{2}}
  \end{equation}
  and finally, applying $|D|$ which commutes both with $\mathcal{D}$
  and $\Lambda_{\omega}$, we have also
  \begin{equation}\label{eq:mixedestangH1}
    \left\|\Lambda_{\omega}^{s}
      \int_{0}^{t} e^{i(t-t') \mathcal{D}}F(t')dt'
    \right\|_{L^{\infty}\dot H^{1}}\lesssim
    \|\bra{x}^{\frac12+}
       \Lambda_{\omega}^{s}|D|F\|_{L^{2}L^{2}}    
  \end{equation}
  Now we use again the representation \eqref{eq:represI-II};
  by \eqref{eq:energyangfree} and \eqref{eq:mixedestangH1}
  we can write
  \begin{equation*}
    \|\Lambda_{\omega}^{s}
      e^{it(\mathcal{D}+V)}f\|_{L^{\infty}\dot H^{1}}
    \lesssim
    \|\Lambda_{\omega}^{s}f\|_{H^{1}}+
    \|\bra{x}^{\frac12+}\Lambda_{\omega}^{s}|D|(Vu)\|_{L^{2}L^{2}}
  \end{equation*}
  and proceeding exactly as in the first part of the proof
  we arrive at \eqref{eq:energyang}.
\end{proof}

% The operator $\mathcal{D}+V$ is selfadjoint under condition
% \eqref{eq:hardyV} (see \cite{Thaller92-a}). Moreover, by Hardy's
% inequality in 3D $\||x|^{-1}f\|_{L^{2}}\le C\|\nabla f\|_{L^{2}}$
% we see that
% \begin{equation*}
%   \|Vf\|_{L^{2}}\le C \cdot \delta\|\nabla f\|_{L^{2}}
%   \le C' \cdot \delta\|\mathcal{D}f\|_{L^{2}}\qquad
%   (\|\mathcal{D}f\|_{L^{2}}\simeq\|\nabla f\|_{L^{2}})
% \end{equation*}
% which implies \eqref{eq:equivhardy}. Since $\mathcal{D}+V$
% commutes with the flow $e^{it(\mathcal{D}+V)}$, \eqref{eq:energyH1}
% is an immediate consequence of \eqref{eq:equivhardy} and the
% conservation of $L^{2}$ norm.

 % section the_dirac_equation_with_potential (end)

\section{The nonlinear Dirac equation}\label{sec:NLD}  %(fold)

Theorem \ref{the:strichD} contains all the necessary tools to
prove global well posedness for the cubic nonlinear Dirac
equation
\begin{equation}\label{eq:NLD}
  iu_{t}=\mathcal{D}u+Vu+P_{3}(u,\overline{u}),\qquad
  u(0,x)=f(x).
\end{equation}
Our result is the following:

\begin{theorem}\label{the:globalNL}
  Consider the perturbed Dirac system \eqref{eq:NLD}, where
  the $4\times4$ matrix valued potential $V=V(|x|)$
  is hermitian and satisfies assumptions \eqref{eq:nablaangV2}.
  Let $P_{3}(u,\overline{u})$ be a $\mathbb{C}^{4}$-valued homogeneous
  cubic polynomial.
  Then for any $s>1$ there exists $\epsilon_{0}$ such that
  for all initial data satisfying
  \begin{equation}\label{eq:data}
    \|\Lambda_{\omega}^{s}f\|_{H^{1}}<\epsilon_{0}
  \end{equation}
  the Cauchy problem \eqref{eq:NLD} admits a unique global solution
  $u\in CH^{1}\cap L^{2}L^{\infty}$
  with $\Lambda^{s}_{\omega}u\in L^{\infty}H^{1}$.
\end{theorem}

\begin{proof}%[of ...]
  The proof is based on a fixed point argument in the space
  $X$ defined by the norm
  \begin{equation}\label{eq:spaceX}
    \|u\|_{X}:=
    \|\Lambda^{s}_{\omega}u\|_{L^{2}_{t}L^{\infty}_{r}L^{2}_{\omega}}
    +
    \|\Lambda^{s}_{\omega}u\|_{L^{\infty}_{t}H^{1}_{x}}.
  \end{equation}
  Notice that in Theorem \ref{the:strichD} we proved the estimate
  \begin{equation}\label{eq:estX}
    \|e^{it(\mathcal{D}+V)}f\|_{X}\lesssim
    \|\Lambda^{s}_{\omega}\|_{H^{1}}.
  \end{equation}
  Define $u=\Phi(v)$ for $v\in X$ as the solution of the
  linear problem
  \begin{equation}\label{eq:linearized}
    iu_{t}=\mathcal{D}u+Vu+P(v,\overline{v}),\qquad
    u(0,x)=f(x)
  \end{equation}
  and represent $u$ as
  \begin{equation*}
    u=\Phi(v)=
    e^{it(\mathcal{D}+V)}f+
      i \int_{0}^{t}e^{i(t-t')(\mathcal{D}+V)}
         P(v(t'),\overline{v(t')})dt'.
  \end{equation*}
  % Recall that the operator 
  %   \begin{equation*}
  %     \mathcal{D}_V=\mathcal{D}+V
  %   \end{equation*}
  %   is selfadjoint under
  %   the stated assumptions on the potential 
  %   (see \cite{DanconaFanelli07-a}, \cite{DanconaFanelli08-a}), thus
  %   the propagator $e^{it\mathcal{D}_V}$ is well defined and unitary 
  %   on $L^{2}$, and we
  %   need only show that $\Phi$ is a contraction on a suitable bounded
  %   subset of $X$. 
  We recall now the product estimate
  \begin{equation*}%\label{eq:prodest}
    \|\Lambda^{s}_{\omega}(gh)\|_{L^{2}_{\omega}(\mathbb{S}^{2})}
    \lesssim
    \|\Lambda^{s}_{\omega}g\|_{L^{2}_{\omega}(\mathbb{S}^{2})}
    \|\Lambda^{s}_{\omega}h\|_{L^{2}_{\omega}(\mathbb{S}^{2})}
  \end{equation*}
  (see \eqref{eq:prodest}). Then we have, by \eqref{eq:estX}
  \begin{equation*}
  \begin{split}
    \|u\|_{X}
    &\lesssim
    \|\Lambda^{s}_{\omega}f\|_{H^{1}}+
    \int_{0}^{\infty}
    \|e^{i(t-t')\mathcal{D}}P(v(t'),\overline{v(t')})\|_{X}
    dt'
    \\
    & \lesssim
    \|\Lambda^{s}_{\omega}f\|_{H^{1}}+
    \int_{0}^{\infty}
     \|\Lambda^{s}_{\omega}
      P(v(t'),\overline{v(t')})\|_{H^{1}}dt'
    \equiv
    \|\Lambda^{s}_{\omega}f\|_{H^{1}}+
    \|\Lambda^{s}_{\omega}
      P(v,\overline{v})\|_{L^{1}H^{1}}.
  \end{split}
  \end{equation*}
  By \eqref{eq:prodest} we have
  \begin{equation*}
    \|\Lambda_{\omega}^{s}(v^{3})\|_{L^{2}_{\omega}(\mathbb{S}^{2})}
    \lesssim
    \|\Lambda_{\omega}^{s}v\|^{3}_{L^{2}_{\omega}(\mathbb{S}^{2})}
  \end{equation*}
  whence
  \begin{equation*}
    \|\Lambda_{\omega}^{s}(v^{3})\|_{L^{2}_{x}}
    \lesssim
    \|\Lambda_{\omega}^{s}v\|_{L^{2}_{x}}
    \|\Lambda_{\omega}^{s}v\|^{2}_{L^{\infty}_{r} L^{2}_{\omega}}
  \end{equation*}
  and
  \begin{equation}\label{eq:prod1}
    \|\Lambda_{\omega}^{s}(v^{3})\|_{L^{1}_{t}L^{2}_{x}}
    \lesssim
    \|\Lambda_{\omega}^{s}v\|_{L^{\infty}_{t} L^{2}_{x}}
    \|\Lambda_{\omega}^{s}v\|^{2}
        _{L^{2}_{t}L^{\infty}_{r} L^{2}_{\omega}}
    \le \|v\|_{X}^{3}.
  \end{equation}
  In a similar way,
  \begin{equation*}
    \|\Lambda_{\omega}^{s}\nabla(v^{3})\|
          _{L^{2}_{\omega}(\mathbb{S}^{2})}
    \lesssim
    \|\Lambda_{\omega}^{s}\nabla v\|
        _{L^{2}_{\omega}(\mathbb{S}^{2})}
    \|\Lambda_{\omega}^{s}v\|^{2}_{L^{2}_{\omega}(\mathbb{S}^{2})}
  \end{equation*}
  so that
  \begin{equation*}
    \|\Lambda_{\omega}^{s}\nabla(v^{3})\|
          _{L^{2}_{x}}
    \lesssim
    \|\Lambda_{\omega}^{s}\nabla v\|
        _{L^{2}_{x}}
    \|\Lambda_{\omega}^{s}v\|^{2}
        _{L^{\infty}_{r} L^{2}_{\omega}}
  \end{equation*}
  and
  \begin{equation}\label{eq:prod2}
    \|\Lambda_{\omega}^{s}\nabla(v^{3})\|_{L^{1}_{t}L^{2}_{x}}
    \lesssim
    \|\Lambda_{\omega}^{s}\nabla v\|_{L^{\infty}_{t} L^{2}_{x}}
    \|\Lambda_{\omega}^{s}v\|^{2}
        _{L^{2}_{t}L^{\infty}_{r} L^{2}_{\omega}}
    \le \|v\|_{X}^{3}.
  \end{equation}
  In conclusion, \eqref{eq:prod1} and \eqref{eq:prod2} imply
  \begin{equation*}
     \|\Lambda^{s}_{\omega}
        P(v,\overline{v})\|_{L^{1}H^{1}}\lesssim
    \|v\|^{3}_{X}
  \end{equation*}
  and the estimate for $u=\Phi(v)$ is
  \begin{equation*}
    \|u\|_{X}\equiv\|\Phi(v)\|_{X}\lesssim
    \|\Lambda^{s}_{\omega} f\|_{H^{1}}+\|v\|_{X}^{3}.
  \end{equation*}
  An analogous computation gives the estimate
  \begin{equation*}
    \|\Phi(v)-\Phi(w)\|_{X}\lesssim
    \|v-w\|_{X}\cdot(\|v\|_{X}+\|w\|_{X})^{2}
  \end{equation*}
  and an application of the contraction mapping theorem
  concludes the proof.
\end{proof}

% \printbibliography
\end{document}